\documentclass[12pt]{amsart}
\usepackage{amsmath,amssymb,amsfonts,amscd}
\usepackage{latexsym}
\usepackage{epic,eepic,epsfig,psfrag}
\usepackage{amscd}
\usepackage{hyperref}
\usepackage{color}
\usepackage{amsbsy}
\usepackage{tensor}
\usepackage{fixmath}
\usepackage{tikz-cd}

\let\oldtocsection=\tocsection
 
\let\oldtocsubsection=\tocsubsection
 
\let\oldtocsubsubsection=\tocsubsubsection
 
\renewcommand{\tocsection}[2]{\hspace{0em}\oldtocsection{#1}{#2}}
\renewcommand{\tocsubsection}[2]{\hspace{2em}\oldtocsubsection{#1}{#2}}
\renewcommand{\tocsubsubsection}[2]{\hspace{4em}\oldtocsubsubsection{#1}{#2}}

\newcommand\ul{\underline}

\def\k{\kappa}

\def\s{\sigma}
\def\t{\tau}
\def\i{\iota}

\def\G{\Gamma}

\def\cX{{\mathcal X}}
\def\cY{{\mathcal Y}}

\newtheorem{dfn}{Definition}[section]
\newtheorem{lem}[dfn]{Lemma}
\newtheorem{prp}[dfn]{Proposition}
\newtheorem{thm}[dfn]{Theorem}
\newtheorem{rmk}[dfn]{Remark}

\newtheorem{ex}[dfn]{Example}

\author{Gunnar Carlsson, Benjamin Filippenko}

\title{Persistent homology of the sum metric}

\makeatletter
\def\blfootnote{\xdef\@thefnmark{}\@footnotetext}
\makeatother

\begin{document}

\begin{abstract}
Given finite metric spaces $(X, d_X)$ and $(Y, d_Y)$, we investigate the persistent homology $PH_*(X \times Y)$ of the Cartesian product $X \times Y$ equipped with the sum metric $d_X + d_Y$. Interpreting persistent homology as a module over a polynomial ring, one might expect the usual K\"unneth short exact sequence to hold. We prove that it holds for $PH_0$ and $PH_1$, and we illustrate with the Hamming cube $\{0,1\}^k$ that it fails for $PH_n,\,\, n \geq 2$. For $n = 2$, the prediction for $PH_2(X \times Y)$ from the expected K\"unneth short exact sequence has a natural surjection onto $PH_2(X \times Y)$. We compute the nontrivial kernel of this surjection for the splitting of Hamming cubes $\{0,1\}^k = \{0,1\}^{k-1} \times \{0,1\}$. For all $n \geq 0$, the interleaving distance between the prediction for $PH_n(X \times Y)$ and the true persistent homology is bounded above by the minimum of the diameters of $X$ and $Y$. As preliminary results of independent interest, we establish an algebraic K\"unneth formula for simplicial modules over the ring $\k[\mathbb{R}_+]$ of polynomials with coefficients in a field $\k$ and exponents in $\mathbb{R}_+ = [0,\infty)$, as well as a K\"unneth formula for the persistent homology of $\mathbb{R}_+$-filtered simplicial sets -- both of these K\"unneth formulas hold in all homological dimensions $n \geq 0$.
\end{abstract}

\maketitle

\tableofcontents

\blfootnote{2010 Mathematics Subject Classification. Primary 55N99, 55U25.}
\blfootnote{This material is based upon work supported by the National Science Foundation under Award No.\ 1708916 and No.\ 1903023.}
\newpage

\section{Introduction} \label{sec:kunnethintroduction}

Topological Data Analysis (TDA) aims to understand the topology of an ambient space from a finite sample; see \cite{carlsson2014topological} for an overview of the field. We assume that the topology of the ambient space is induced by a metric $d$, making the finite sample $X$ into a finite metric space $(X,d)$. In applications, the ambient space is often embedded in $\mathbb{R}^n$ and various metrics $d$ on $\mathbb{R}^n$ can be useful.

A central tool in TDA is the \emph{persistent homology} $PH_*(X)$ of the Vietoris-Rips complex associated to a finite metric space $(X,d)$. Persistent homology provides approximations of the homology of the ambient space from which $X$ was sampled, and it does so at all scales $t \geq 0$. The scale parameter $t$ takes all values $t \in \mathbb{R}_+ = [0,\infty)$, and for each fixed value it upper bounds the distances in $X$ that are `seen at scale $t$.'

Many applications now use persistent homology for feature generation. This idea is used, for example, in the study of databases of molecules \cite{fullrene}.  The results of the present paper give relationships among such features, which can be expected to be useful as the applications of the method become more sophisticated. 

In this paper, given two finite metric spaces $(X, d_X)$ and $(Y, d_Y)$, we equip the Cartesian product $X \times Y$ with the sum metric $d_X + d_Y$. We are interested in computing $PH_*(X \times Y)$ in terms of $PH_*(X)$ and $PH_*(Y)$. A K\"unneth formula of this type would allow us to compute persistent homology of interesting spaces, such as the Hamming cube $I^k$ for $k \geq 1,\, I = \{0,1\}$ with the Hamming distance; see Example~\ref{ex:intro-hamming} and Section~\ref{sec:hammingcube} for a discussion of the Hamming cube. In \cite{GP}, the authors derive a K\"unneth formula for the maximum metric $d_{X \times Y} := \max\{d_X,d_Y\}$ on $X \times Y$ which holds in all homological dimensions.

Our main result (Theorem~\ref{thm:intro-maintheorem}) is that the familiar K\"unneth short exact sequence holds for the sum metric $d_{X \times Y} := d_X + d_Y$ in homological dimensions $0$ and $1$, and in dimension~$2$ it computes a module that admits a surjection onto $PH_2(X \times Y)$; see Theorem~\ref{thm:kunnethmain} for a more refined statement and a proof. Example~\ref{ex:intro-hamming} below shows that the short exact sequence fails in homological dimension $n \geq 2$. We discuss the underlying reason for this failure below (both in this introduction and in the beginning of Section~\ref{sec:metricspaces}). Here we view persistent homology $PH_*(X)$ as a module over a ring $\k[\mathbb{R}_+]$ of polynomials in a single variable with exponents in $\mathbb{R}_+ = [0,\infty)$ and coefficients in a field $\k$; see Section~\ref{sec:modules} for a study of $\k[\mathbb{R}_+]$-modules, and in particular see Remark~\ref{rmk:artinrees} for a discussion of their relationship with persistence vector spaces.

\begin{thm} \label{thm:intro-maintheorem}
Consider finite metric spaces $(X,d_X)$ and $(Y,d_Y)$ and the product space equipped with the sum metric $(X \times Y, d_X + d_Y)$. Then for $n = 0,1$ there is a short exact sequence
\begin{align*}
0 &\rightarrow \bigoplus_{i+j = n} PH_i(X) \otimes_{\k[\mathbb{R}_+]} PH_j(Y) \rightarrow PH_n(X \times Y)\\
&\rightarrow \bigoplus_{i+j =n-1} Tor_1(PH_i(X),PH_j(Y)) \rightarrow 0
\end{align*}
which is natural with respect to distance non-increasing maps $(X,d_X) \rightarrow (X',d_{X'})$ and $(Y,d_Y) \rightarrow (Y',d_{Y'})$. This sequence splits, but not naturally. Moreover, in dimension $n = 2$, the direct sum of the term in the left position (tensor product terms) and the term in the right position (Tor terms) admits a surjection onto $PH_2(X \times Y)$.
\hfill$\square$
\end{thm}

\begin{rmk}
To compute the bars in $PH_n(X \times Y)$ for $n = 0,1$ using the split short exact sequence in Theorem~\ref{thm:intro-maintheorem}, it suffices to use the formulas for tensor product and Tor of bars computed in Proposition~\ref{prp:tensortorbars}. See Section~\ref{subsec:bars} for an explanation of bars.
\end{rmk}

\begin{rmk}
Theorem~\ref{thm:intro-maintheorem} and all other results about metric spaces $(X,d)$ in this paper hold for a more general class of spaces: Precisely, it is required that the pairing $d : X \times X \rightarrow \mathbb{R}_+$ is symmetric $d(x,y) = d(y,x)$, but $d$ is not required to be reflexive $d(x,y) = 0 \iff x = y$ or satisfy the triangle inequality $d(x,z) \leq d(x,y) + d(y,z).$
\end{rmk}

For all $n \geq 0$, denote the prediction for $PH_n(X \times Y)$ from the short exact sequence in Theorem~\ref{thm:intro-maintheorem} by $PH_n(X,Y)$. That is, $PH_n(X,Y)$ is isomorphic to the direct sum of the tensor product terms on the left and the Tor terms on the right. The following result bounds the interleaving distance between the homology $PH_n(X \times Y)$ and the prediction  $PH_n(X,Y)$ from above by the minimum of the diameters of $X$ and $Y$; see Theorem~\ref{thm:interleavingdistance} for the proof. We discuss the interleaving distance in Section~\ref{sec:interleavingdistance}. The idea to look for a bound on this distance was suggested by Leonid Polterovich.

\begin{thm} \label{thm:interleavingdistance-intro}
The interleaving distance between $PH_*(X,Y)$ and $PH_*(X \times Y)$ is less than or equal to the minimum of the diameters
$$\min(\text{\rm diameter}(X),\text{\rm diameter}(Y)),$$
where the diameter is the maximum distance
$$\text{diameter}(X) = \max \{d_X(x_0,x_1) \,\, | \,\, x_0,x_1 \in X\}.$$
\hfill$\square$
\end{thm}

We now outline the structure of the paper, including preliminary K\"unneth type results used in the proof of Theorem~\ref{thm:intro-maintheorem} which are of independent interest. These other K\"unneth formulas hold in all homological dimensions $n \geq 0$. We also summarize why Theorem~\ref{thm:intro-maintheorem} holds, and why it fails in homological dimensions $n \geq 2$.

In Section~\ref{sec:modules} we study finitely presented $\k[\mathbb{R}_+]$-modules and simplicial modules (Defintion~\ref{dfn:simplicialmodule}). Given a simplicial $\k[\mathbb{R}_+]$-module $M$, there is an associated chain complex $C_*(M)$; see Definition~\ref{dfn:alternatingfacemapscomplex}. The main result (Theorem~\ref{thm:intro-kunnethformodules}) proved in Section~\ref{sec:algebraickunneth} is an algebraic K\"unneth formula for the homology $H_*(C_*(M \otimes_{\k[\mathbb{R}_+]} N))$ for free finitely generated simplicial $\k[\mathbb{R}_+]$-modules $M$ and $N$. See Theorem~\ref{thm:kunnethformodules} for the proof. We remark that, if $\k[\mathbb{R}_+]$ were a principal ideal domain (PID), then this theorem would be immediate from the usual K\"unneth short exact sequence for simplicial modules over a PID. However, unlike the usual polynomial ring $\k[\mathbb{Z}_+]$ with exponents in $\mathbb{Z}_+= \{0,1,2,\ldots \}$, the ring $\k[\mathbb{R}_+]$ is not a PID. Related K\"unneth formulas have been proven in \cite{MR3734662} and \cite{BMpersistentkunneth}.

\begin{thm} \label{thm:intro-kunnethformodules} {\bf (Algebraic Persistent K\"unneth Theorem)}
Let $M,N$ be $\mathbb{R}_+$-graded free finitely generated simplicial $\k[\mathbb{R}_+]$-modules. Then for $n \geq 0$ there is a short exact sequence
\begin{align*}
0 &\rightarrow \bigoplus_{i + j = n} H_i(C_*(M)) \otimes_{\k[\mathbb{R}_+]} H_j(C_*(N)) \rightarrow H_n(C_*(M \otimes_{\k[\mathbb{R}_+]} N))\\
&\rightarrow \bigoplus_{i+j = n-1} Tor_1(H_i(C_*(M)), H_j(C_*(N))) \rightarrow 0
\end{align*}
which is natural in both $M$ and $N$. Moreover, the sequence splits, but not naturally.
\hfill$\square$
\end{thm}

Building on this algebraic K\"unneth formula over $\k[\mathbb{R}_+]$, in Section~\ref{sec:filteredsimplicialsets} we establish a K\"unneth formula (Theorem~\ref{thm:kunnethfilteredsimplicialsets}) for persistent homology of $\mathbb{R}_+$-filtered simplicial sets $({\bf X}, l)$ (Definitions~\ref{dfn:filteredsimplicialset},\,\ref{dfn:persistenthomologyfilteredsimplicialset}). These are simplicial sets ${\bf X}$ (see Section~\ref{sec:simplicialsets} for a review of simplicial sets) together with a map $l : {\bf X}_n \rightarrow \mathbb{R}_+$ on the $n$-simplices ${\bf X}_n$ of ${\bf X}$ for each $n \geq 0$ such that $l$ is non-increasing along the structure maps (face and degeneracy) in the simplicial set. For any $t \in \mathbb{R}_+$, the collection of simplices $\s \in {\bf X}$ such that $l(\s) \leq t$ forms a simplicial set ${\bf X}^{(t)}$, and if $t \leq t'$ then ${\bf X}^{(t)} \subseteq {\bf X}^{(t')}$. The persistent homology $PH_*({\bf X},l)$ is a $\mathbb{R}_+$-graded $\k[\mathbb{R}_+]$-module with homogeneous degree-$t$ part equal to the homology of the simplicial set ${\bf X}^{(t)}$, i.e. $PH_*^{(t)}({\bf X},l) = H_*({\bf X}^{(t)};\k).$

The product of filtered simplicial sets $({\bf X}, l_X)$ and $({\bf Y}, l_Y)$ is the pair $({\bf X} \times {\bf Y}, l_X + l_Y)$ where ${\bf X} \times {\bf Y}$ is the usual product of simplicial sets and $l_X + l_Y$ is the pointwise sum of the filtration functions. In the notation of the above paragraph, this corresponds to the filtration $({\bf X} \times {\bf Y})^{(t)} = \bigcup_{t_X + t_Y \leq t} {\bf X}^{(t_X)} \times {\bf Y}^{(t_Y)}$ for $t \in \mathbb{R}_+$. The following K\"unneth theorem follows from Theorem~\ref{thm:intro-kunnethformodules}; see Section~\ref{sec:filteredsimplicialsets} for the proof. In \cite{GP}, the authors derive K\"unneth formulas for the persistent homology of the product $\cX \times \cY$ of $\mathbb{Z}_+$-filtered topological spaces $\cX = (\cX^{(0)} \subset \cX^{(1)} \subset \cX^{(2)} \subset \cdots )$ for various choices of filtration on $\cX \times \cY$, including $(\cX \times \cY)^{(t)} = \bigcup_{t_X + t_Y \leq t} \cX^{(t_X)} \times \cY^{(t_Y)}$ for $t \in \mathbb{Z}_+$. 

\begin{thm} \label{thm:kunnethfilteredsimplicialsets} {\bf (K\"unneth Theorem for $\mathbb{R}_+$-filtered simplicial sets)}
Given $\mathbb{R}_+$-filtered simplicial sets $({\bf X},l_X)$ and $({\bf Y},l_Y)$ such that ${\bf X}_n$ and ${\bf Y}_n$ are finite sets for all $n \geq 0$, then for $n \geq 0$ there is a short exact sequence
\begin{align*}
0 &\rightarrow \bigoplus_{i+j = n} PH_i({\bf X},l_X) \otimes_{\k[\mathbb{R}_+]} PH_j({\bf Y},l_Y) \rightarrow PH_n({\bf X} \times {\bf Y}, l_X + l_Y)\\
&\rightarrow \bigoplus_{i+j =n-1} Tor_1(PH_i({\bf X},l_X),PH_j({\bf Y},l_Y)) \rightarrow 0
\end{align*}
which is natural with respect to maps $({\bf X},l_X) \rightarrow ({\bf X}',l_{X'})$ and $({\bf Y},l_Y) \rightarrow ({\bf Y}',l_{Y'})$. Moreover, the sequence splits, but not naturally.
\hfill$\square$
\end{thm}

In Section~\ref{sec:metricspaces} we study the persistent homolgy $PH_*(X)$ of metric spaces $(X, d_X)$. We define $PH_*(X)$ to be the persistent homology $PH_*({\bf X}, l_X)$ of a $\mathbb{R}_+$-filtered simplicial set $({\bf X},l_X)$ associated to the metric space. The filtration function $l_X$ is defined using the metric $d_X$ (see \eqref{eq:maxlength}). The homogeneous degree-$t$ part $PH_*^{(t)}(X)$ is equal to homology $H_*({\bf X}^{(t)};\k)$, which in turn is isomorphic to homology of the time-$t$ Vietoris-Rips simplicial complex $V(X,t)$ since $V(X,t)$ and ${\bf X}^{(t)}$ are homotopy equivalent; see Remark~\ref{rmk:vietorisrips} for Vietoris-Rips $V(X,t)$. That is, we have $PH_*^{(t)}(X) \cong H_*(V(X,t);\k)$
for all $t \in \mathbb{R}_+$.

The product of two metric spaces $X \times Y$ equipped with the sum metric $d_X + d_Y$ has associated $\mathbb{R}_+$-filtered simplicial set $({\bf X} \times {\bf Y}, l_{X \times Y})$, where the filtration function $l_{X \times Y}$ is defined using the sum metric $d_X + d_Y$. The underlying reason for the failure of Theorem~\ref{thm:intro-maintheorem} in dimensions $n \geq 2$ is that the filtrations functions satisfy the inequality (see \eqref{eq:filtrationinequality})
$$l_{X \times Y} \leq l_X + l_Y,$$
but equality does not necessarily hold. The inequality provides a graded module homomorphism
$$PH_n({\bf X} \times {\bf Y}, l_X + l_Y) \rightarrow PH_n({\bf X} \times {\bf Y}, l_{X \times Y}) = PH_n(X \times Y),$$
but since equality does not hold, it is not necessarily an isomorphism. If it were an isomorphism, then Theorem~\ref{thm:intro-maintheorem} would follow immediately from Theorem~\ref{thm:kunnethfilteredsimplicialsets} in all dimensions $n \geq 0$. We show in Section~\ref{sec:les} that this homomorphism fits into a long exact sequence, and in Section~\ref{sec:kunnethinlowdims} we show that the correction term in the long exact sequence vanishes in homological dimensions $n = 0,1,2$. Hence the homomorphism is an isomorphism for $n = 0,1$ and a surjection for $n = 2$. Theorem~\ref{thm:intro-maintheorem} then follows from Theorem~\ref{thm:kunnethfilteredsimplicialsets}.

We remark that, in the notation of Theorem~\ref{thm:interleavingdistance-intro}, we have
$$PH_n(X,Y) = PH_n({\bf X} \times {\bf Y}, l_X + l_Y)$$
by definition; see \eqref{eq:predictedmodule} and Proposition~\ref{prp:kunnethdiagonal}.

In Section~\ref{sec:hammingcube} and Example~\ref{ex:intro-hamming}, we investigate $PH_*(I^k)$ where $I^k = \{0,1\}^k$ is the Hamming cube with the Hamming metric. We use the splitting $I^k = I^{k-1} \times I$ of the Hamming metric to apply Theorem~\ref{thm:intro-maintheorem} inductively to compute $PH_n(I^k)$ for $n = 0,1,2,$ and for all $k \geq 1$. This involves a technical computation of the nontrivial kernel of the surjection $PH_2(I \times I^{k-1}, l_I + l_{I^{k-1}}) \rightarrow PH_2(I^k)$; see Propositions~\ref{prp:ph1lemma},\,\ref{prp:ph2lemma}.

\begin{ex} \label{ex:intro-hamming}
Consider the metric space $I = \{0,1\}$ with distance $1$ between its two points. The $k$-dimensional {\bf Hamming Cube} is the $k$-fold Cartesian product $I^k$ consisting of $k$-tuples of zeros and ones equipped with the Hamming metric: The distance between two $k$-tuples in $I^k$ is the number of coordinates in which they are not equal. The Hamming metric is equal to the sum metric on $I^k$ given by summing the metric on each of the $k$ factors of $I$, which is also equal to the sum metric on the two factor splitting $I^k = I^{k-1} \times I$. So, from a full K\"unneth formula for the sum metric on a product $X \times Y$, one hopes to inductively compute $PH_n(I^k)$ for all $n$ and $k$.

In Section~\ref{sec:hammingcube} we prove that $PH_2(I^k) = 0$ for all $k$. This shows that Theorem~\ref{thm:intro-maintheorem} does not hold in dimension $n = 2$. Indeed, for $X = I$ and $Y = I^2$, the product is $X \times Y = I^3$, and the short exact sequence in Theorem~\ref{thm:intro-maintheorem} for $n = 2$ would predict that $PH_2(I^3)$ has $1$ bar; see Section~\ref{subsec:bars} for a review of bars and how to compute tensor product and Tor of bars.

The table below displays the number of bars in $PH_n(I^k)$ for low values of $n$ and $k$. This data shows that Theorem~\ref{thm:intro-maintheorem} does not hold in higher dimensions $n > 2$. Indeed, again for $X = I$ and $Y = I^2$ with product $X \times Y = I^3$, the short exact sequence in Theorem~\ref{thm:intro-maintheorem} for $n = 3$ would predict that $PH_3(I^3) = 0$, but in fact $PH_3(I^3)$ has $1$ bar.
\\

\begin{center}
\begin{tabular}{lclclclclclclcl}
\hline
\multicolumn{8}{|c|}{Number of bars in $PH_n(I^k)$} \\
\hline\\
                            & $k=1$ & $k=2$ & $k=3$ & $k=4$ & $k=5$ & $k=6$ & $k=7$ \\
\hline
$PH_0(I^k)$ & 2   & 4   & 8   & 16  & 32  & 64  & 128 \\
$PH_1(I^k)$ & 0   & 1   & 5   & 17  & 49  & 129 & 321 \\
$PH_2(I^k)$ & 0   & 0   & 0   & 0   & 0   & 0   & 0   \\
$PH_3(I^k)$ & 0   & 0   & 1   & 9   & 49  & 209 & 769 \\
$PH_4(I^k)$ & 0   & 0   & 0   & 0   & 1   & 11  & 71  \\
$PH_5(I^k)$ & 0   & 0   & 0   & 0   & 0   & --- & --- \\
$PH_6(I^k)$ & 0   & 0   & 0   & 0   & 0   & --- & --- \\
$PH_7(I^k)$ & 0   & 0   & 0   & 1   & 10  & --- & ---
\end{tabular}
\end{center}
We discuss the Hamming cube further in Section~\ref{sec:hammingcube}.
\hfill$\square$
\end{ex}

\section{$\k[\mathbb{R}_+]$-modules} \label{sec:modules}

The ring $\k[\mathbb{R}_+]$ is the monoid ring of $\mathbb{R}_+ = [0,\infty)$ over a field $\k$. Concretely, it is the polynomial ring in a single variable $T$ with exponents in $\mathbb{R}_+$ and coefficients in $\k$. There is a natural $\mathbb{R}_+$-grading on $\k[\mathbb{R}_+$] where the homogeneous degree $t \in \mathbb{R}_+$ elements are the monomials with exponent $t$. The usual polynomial ring in a single variable is the monoid ring $\k[\mathbb{Z}_+]$ over the nonnegative integers $\mathbb{Z}_+ = \{0,1,2,\ldots\}$, i.e., the exponents on the variable are in $\mathbb{Z}_+$.

We are interested in the ring $\k[\mathbb{R}_+]$ because the persistent homology $PH_*(X)$ of a metric space $(X,d)$ is a $\k[\mathbb{R}_+]$-module. Indeed, $PH_*(X)$ is often viewed as a persistence vector space -- or in other words a functor from the poset $\mathbb{R}_+$ to the category of $\k$-vector spaces -- and a persistence vector space is equivalent to a $\mathbb{R}_+$-graded $\k[\mathbb{R}_+]$-module; see Remark~\ref{rmk:artinrees}. For the purposes of this paper, it is more convenient to work with $\k[\mathbb{R_+}]$-modules.

Our main result in this section in an algebraic K\"unneth formula (Theorem~\ref{thm:kunnethformodules}) for simplicial modules over $\k[\mathbb{R}_+]$ (Section~\ref{sec:simplicialmodules}). Theorem~\ref{thm:kunnethformodules} would be immediate from the usual K\"unneth short exact sequence for simplicial modules over a principal ideal domain -- but $\k[\mathbb{R}_+]$ is not a principal ideal domain, unlike the usual polynomial ring $\k[\mathbb{Z}_+]$ with integer exponents. Indeed, $\k[\mathbb{R}_+]$ is not even Noetherian: Consider the ascending chain of principle ideals $(T) \subset (T^{1/2}) \subset (T^{1/3}) \subset \cdots.$ We resolve these issues in Section~\ref{sec:finiteness}: The ring $\k[\mathbb{R}_+]$ has the crucial property that the kernel and the image of a homomorphism between free finitely generated modules are both free and finitely generated; see Lemma~\ref{lem:finitelygeneratedfree}. This property of $\k[\mathbb{R}_+]$ is enough to establish Theorem~\ref{thm:kunnethformodules}.

In Section~\ref{subsec:bars}, we discuss the barcode classification of finitely presented persistence vector spaces in the language of finitely presented $\k[\mathbb{R}_+]$-modules (Proposition~\ref{prp:classification}), and we compute the tensor product and Tor of bars in Proposition~\ref{prp:tensortorbars}. This allows for efficient computation with the algebraic K\"unneth theorem (Theorem~\ref{thm:kunnethformodules}).

\begin{rmk} \label{rmk:artinrees} {\bf (Persistence vector spaces and Artin-Rees)}
The Artin-Rees construction provides an equivalence between finitely generated $\mathbb{R}_+$-graded $\k[\mathbb{R}_+]$-modules and finitely generated persistence vector spaces over $\k$. We describe this equivalence in this remark. See \cite[Theorem~3.1]{computingpersistenthomologyZC} for a discussion over $\mathbb{Z}_+$ instead of $\mathbb{R}_+$.

Recall the definition of a persistence vector space $V = \{V_t\}_{t \in \mathbb{R}_+}$ over $\k$ from \cite[Def.~3.3]{carlsson2014topological}. This is a functor $V : \mathbb{R}_+ \rightarrow Vec_{\k}$, where here $\mathbb{R}_+$ denotes the partially ordered set $\mathbb{R}_+$ viewed as a category and $Vec_{\k}$ is the category of vector spaces over $\k$. Let $PVec_{\k}$ denote the category of persistence vector spaces over $\k$.

The Artin-Rees construction provides an equivalence of categories $A : PVec_{\k} \rightarrow \k[\mathbb{R}_+]\text{-Mod}$, where $\k[\mathbb{R}_+]\text{-Mod}$ is the category of graded $\k[\mathbb{R}_+]$-modules. Explicitly, $A(V)$ is the $\k$-vector space $\oplus_{t \in \mathbb{R}_+} V_t$. The $\k[\mathbb{R}_+]$-module structure on $A(V)$ is given as follows. Let $T$ denote the variable in $\k[\mathbb{R}_+]$, and for $t \leq t'$ let $l_{t,t'} : V_t \rightarrow V_{t'}$ denote the corresponding map in $V$. Then for a homogeneous element $v \in V_t$ we define $T^a \cdot v = l_{t,t+a}(v) \in V_{t+a}$ for $a \in \mathbb{R}_+$. This gives $A(V)$ the structure of a $\mathbb{R}_+$-graded $\k[\mathbb{R}_+]$-module.

The inverse functor $A^{-1} : \k[\mathbb{R}_+]\text{-Mod} \rightarrow PVec_{\k}$ is as follows. For $M \in \k[\mathbb{R}_+]\text{-Mod}$ and $t \in \mathbb{R}_+$, the $\k$-vector space $A^{-1}(M)_t$ is the homogeneous degree $t$ part of $M$. For $t \leq t'$ the linear map $l_{t,t'} : A^{-1}(M)_t \rightarrow A^{-1}(M)_{t'}$ is multiplication by $T^{t' - t}$.
\end{rmk}

\subsection{Finiteness and Freeness} \label{sec:finiteness}
Over the ring $\k[\mathbb{R}_+]$, submodules of finitely generated free modules are not necessarily finitely generated free: Indeed, the ideal in $\k[\mathbb{R}_+]$ consisting of all polynomials with zero constant term is neither free nor finitely generated. This is in stark contrast to the polynomial ring with integer exponents $\k[\mathbb{Z}_+]$ which is a principal ideal domain and hence does have the property that submodules of finitely generated free modules are always finitely generated free.

To prove the algebraic K\"unneth formula over $\k[\mathbb{R}_+]$ (Theorem~\ref{thm:kunnethformodules}) and to describe homology using bars via the classification theorem for finitely presented $\k[\mathbb{R}_+]$-modules (Proposition~\ref{prp:classification}), it is essential that, given a homomorphism of finitely generated graded free $\k[\mathbb{R}_+]$-modules, the kernel and the image are finitely generated and graded free. This is the content of Lemma~\ref{lem:finitelygeneratedfree}.

\begin{rmk} \label{rmk:equivtoacta}
The following results, mostly in this subsection, are a rephrasing of results in \cite[Sec.~3.4]{carlsson2014topological} from the language of persistence vector spaces to the language of $\k[\mathbb{R}_+]$-modules (see Remark~\ref{rmk:artinrees} for a discussion of the equivalence between these notions): Lemma~\ref{lem:matrixofmap} corresponds to \cite[Prop.~3.7]{carlsson2014topological}. Lemma~\ref{lem:finitelygeneratedfree} corresponds to the essential step in the proof of \cite[Prop.~3.12]{carlsson2014topological}. And in Proposition~\ref{prp:classification}, we reinterpret the classification of finitely presented persistence vector spaces \cite[Prop.~3.12,~3.13]{carlsson2014topological} as a classification of finitely presented $\mathbb{R}_+$-graded $\k[\mathbb{R}_+]$-modules.
\end{rmk}

To prepare for Lemma~\ref{lem:finitelygeneratedfree}, we first understand the matrix of a $\mathbb{R}_+$-graded $\k[\mathbb{R}_+]$-module homomorphism with respect to a basis.

\begin{lem} \label{lem:matrixofmap}
Consider free and finitely generated $\mathbb{R}_+$-graded $\k[\mathbb{R}_+]$-modules $P$ and $Q$ with bases $\{p_1,\ldots,p_n\}$ and $\{q_1,\ldots,q_m\}$, respectively, such that each $p_j$ and $q_i$ is homogeneous.

Then, given a $\mathbb{R}_+$-graded $\k[\mathbb{R}_+]$-module homomorphism $\varphi : P \rightarrow Q$, its matrix $A^{\varphi}$ with respect to the given bases has homogeneous entries satisfying
\begin{equation} \label{eq:mapmatrixcondition}
\deg(A^{\varphi}_{i,j}) = \deg(p_j) - \deg(q_i) \geq 0 \text{ or } A^{\varphi}_{i,j} = 0 \text{ for all } i,j.
\end{equation}

Conversely, any $(m \times n)$-matrix $A$ with homogeneous entries in $\k[\mathbb{R}_+]$ satisfying \eqref{eq:mapmatrixcondition} determines a $\mathbb{R}_+$-graded $\k[\mathbb{R}_+]$-module homomorphism $\varphi^A : P \rightarrow Q$ by $\varphi^A(p_j) = \sum_{i=1}^m A_{i,j} q_i$.

Moreover, the correspondences $\varphi \mapsto A^{\varphi}$ and $A \mapsto \varphi^A$ are inverses of each other.
\end{lem}
\begin{proof}
Consider a $\mathbb{R}_+$-graded $\k[\mathbb{R}_+]$-module homomorphism $\varphi : P \rightarrow Q$. Then, the entries of the matrix $A^{\varphi}$ are defined by $\varphi(p_j) = \sum_{i = 1}^m A^{\varphi}_{i,j} q_i$. Since $\varphi$ is $\mathbb{R}_+$-graded, the image $\varphi(p_j)$ is homogeneous of degree $\deg(p_j)$. Since the $q_i$ form a basis of $Q$, this implies that either $A^{\varphi}_{i,j}q_i$ is $0$ or it is homogeneous of degree $\deg(p_j)$ for all $i$, which implies that either $A^{\varphi}_{i,j} = 0$ or $\deg(A^{\varphi}_{i,j}) = \deg(p_j) - \deg(q_i)$ for all $i$, as claimed. The converse statement is clear.
\end{proof}

\begin{lem} \label{lem:finitelygeneratedfree}
Suppose $\varphi : P \rightarrow Q$ is a homomorphism of $\mathbb{R}_+$-graded free\footnote{Graded free means that there exists a basis of homogeneous elements.} finitely generated $\k[\mathbb{R}_+]$-modules.

Then, there exist bases $\{p_1,\ldots,p_n\}$ of $P$ and $\{q_1,\ldots,q_m\}$ of $Q$ consisting of homogeneous elements such that the matrix of $\varphi$ with respect to these bases has the property that every row and every column has at most one nonzero entry.

In particular, the kernel and the image of $\varphi$ are graded free and finitely generated.
\end{lem}
\begin{proof}
Since $P$ and $Q$ are graded free and finitely generated, there exists a basis $\{p_1,\ldots,p_n\}$ of $P$ and a basis $\{q_1,\ldots,q_m\}$ of $Q$, both consisting of homogeneous elements. Let $A$ be the matrix of $\varphi$ with respect to these bases. We must show that there exist invertible matrices $B$ and $C$ satisfying \eqref{eq:mapmatrixcondition} such that the matrix $BAC$ has at most one nonzero entry in each row and column; precisely, \eqref{eq:mapmatrixcondition} requires that $B$ satisfies $\deg(B_{i,j}) = \deg(q_j) - \deg(q_i) \geq 0$ or $B_{i,j} = 0$ for all $1 \leq i,j \leq m$ and that $C$ satisfies the same condition with respect to the $p_i$.

Let $q_{i_0}$ be the basis element such that $\deg(q_{i_0})$ is maximal among the degrees of the $q_i$ such that the $i$-th row of $A$ is nonzero, and let $p_{j_0}$ be the basis element such that $\deg(p_{j_0})$ is minimal among the degrees of the $p_j$ such that $A_{i_0,j} \neq 0$. By Lemma~\ref{lem:matrixofmap}, we have 
$$A_{i_0,j_0} = l_0 \cdot T^{\deg(p_{j_0}) - \deg(q_{i_0})}$$
for some $l_0 \in \k$. For any $j \neq j_0$ such that $A_{i_0,j} \neq 0$, we have $A_{i_0,j} = l \cdot T^{\deg(p_j) - \deg(q_{i_0})}$ for some $l \in \k$. There is a column operation given by a matrix $C$ with $C_{j_0,j} = -(l / l_0) \cdot T^{\deg(p_j) - \deg(p_{j_0})}$, ones on the diagonal, and zeros elsewhere, and we have $(AC)_{i_0,j} = 0$. Applying these column operations for all $j \neq j_0$ such that $A_{i_0,j} \neq 0$ arrives at a matrix with $i_0$-th row equal to $0$ except for the $(i_0,j_0)$-entry. Rename $C$ to be the product of all these column operations, and set $A' = AC$. Now, for any $i \neq i_0$ such that $A'_{i,j_0} \neq 0$, we have $A'_{i,j_0} = l \cdot T^{\deg(p_{j_0}) - \deg(q_i)}$ for some $l \in \k$. There is a row operation given by a matrix $B$ with $B_{i,i_0} = -(l/l_0) \cdot T^{\deg(q_{i_0}) - \deg(q_i)}$, ones on the diagonal, and zeros elsewhere, and we have $(BA')_{i,j_0} = 0.$ Applying these row operations for all $i \neq i_0$ such that $A'_{i,j_0} \neq 0$ arrives at a matrix with $j_0$-th column equal to $0$ except for the $(i_0,j_0)$-entry. Rename $B$ to be the product of all these row operations. Then the matrix $BAC$ has the property that the $i_0$-th row and the $j_0$-th column are zero except for the $(i_0,j_0)$-entry.

We may now iterate this process to produce the desired result. Indeed, relabel $BAC$ by $A$, let $q_{i_1}$ denote the basis element such that $\deg(q_{i_1})$ is maximal among the degrees of the $q_i \neq q_{i_0}$ such that the $i$-th row of $A$ is nonzero, and let $p_{j_1}$ denote the basis element such that $\deg(p_{j_1})$ is minimal among the degrees of the $p_j \neq p_{j_0}$ such that $A_{i_1,j} \neq 0$. Then the row and column operations as described above that zero out all entries in the $i_1$-row and $j_1$-column (except for the $(i_1,j_1)$-entry) do not affect the $i_0$-row and the $j_0$-column. Repeating this process completes the construction.
\end{proof}

\subsection{Bars} \label{subsec:bars}

A common notation for finitely presented $\k[\mathbb{R}_+]$-modules uses bars $(a,b)$ for $0 \leq a \leq b \leq \infty$, where $(a,b)$ is shorthand for the module generated by a single generator $\s$ of degree $a$ such that $T^{b-a}\cdot \s = 0$; precisely,
$$(a,b) = \G^a\k[\mathbb{R}_+]/T^{b-a},$$
where $\G^a$ denotes a degree shift by $a$. We allow $b = \infty$, in which case
$$(a,\infty) = \G^a\k[\mathbb{R}_+].$$

The following classification result is equivalent to the classification of finitely presented persistence vector spaces in \cite[Prop.~3.12,~3.13]{carlsson2014topological}; see Remark~\ref{rmk:equivtoacta}. The essential step in the setting of $\k[\mathbb{R}_+]$-modules is Lemma~\ref{lem:finitelygeneratedfree}.

\begin{prp} \label{prp:classification}
Every $\mathbb{R}_+$-graded finitely presented $\k[\mathbb{R}_+]$-module $M$ is isomorphic to a direct sum of the form
$$M \cong \oplus_{i=1}^n (a_i,b_i),$$
where $0 \leq a_i < b_i \leq \infty$. Moreover, this decomposition is unique, i.e. if $M$ is isomorphic to another such direct sum then the number $n$ of bars is the same and the set of pairs $(a_i,b_i)$ that occur, with multiplicities, is the same.
\hfill$\square$
\end{prp}

Due to the above classification, for many computations it suffices to understand how single bars interact with each other. In this paper we use the following computations of the tensor product and $Tor$ of bars; equivalent computations appear in \cite{BMpersistentkunneth}\cite{MR3734662}.

\begin{prp} \label{prp:tensortorbars}
Consider the bars $(a,b)$ and $(c,d)$. Then there are isomorphisms of $\mathbb{R}_+$-graded $\k[\mathbb{R}_+]$-modules
\begin{equation} \label{eq:tensorbars}
(a,b) \otimes_{\k[\mathbb{R}_+]} (c,d) \cong (a+c, \min\{a+d,b+c\})\\
\end{equation}
and
\begin{equation} \label{eq:torbars}
Tor_1((a,b),(c,d)) \cong (\max \{ a+d, b+c\}, b+d).
\end{equation}
Moreover, for any $\mathbb{R}_+$-graded finitely presented $\k[\mathbb{R}_+]$-modules $M$ and $N$, we have $Tor_n(M,N) = 0$ for $n \geq 2$.
\end{prp}
\begin{proof}
We compute $(a,b) \otimes_{\k[\mathbb{R}_+]} (c,d) = \G^a\k[\mathbb{R}_+]/T^{b-a} \otimes \G^c\k[\mathbb{R}_+]/T^{d-c} \cong\\ \G^{a+c}\k[\mathbb{R}_+]/ \langle T^{b-a}, T^{d-c} \rangle = \G^{a+c}\k[\mathbb{R}_+]/ T^{\min\{b-a, d-c\}} = (a+c, a+c + \min\{b-a, d-c\}) = (a+c, \min\{a+d, b+c\}),$ as claimed.

Next, we compute \eqref{eq:torbars}. A projective resolution of $(a,b)$ is given by the exact sequence
\begin{equation} \label{eq:resolutionbar}
0 \rightarrow (b,\infty) \rightarrow (a,\infty) \rightarrow (a,b) \rightarrow 0.
\end{equation}
Removing the $(a,b)$ term and tensoring with $(c,d)$ yields the complex
$$0 \rightarrow (b+c,b+d) \rightarrow (a+c,a+d) \rightarrow 0.$$
The homology at the $(b+c,b+d)$ term is the kernel of the map to $(a+c,a+d)$. If $b+c \geq a+d$ then the map is $0$ and hence the homology is $(b+c,b+d)$. If $b+c \leq a+d$ then the kernel is $(a+d,b+d)$. So, in both cases, the homology is $(\max\{a+d,b+c\},b+d)$, proving \eqref{eq:torbars}.

To prove the final statement of the proposition, first apply Proposition~\ref{prp:classification} to both $M$ and $N$ to write them as direct sums of bars. Then the statement follows since the projective resolution \eqref{eq:resolutionbar} is zero to the left of the $(b,\infty)$ term.
\end{proof}

\subsection{Simplicial $\k[\mathbb{R}_+]$-modules} \label{sec:simplicialmodules}

Let $\k[\mathbb{R}_+]$-Mod denote the category with objects the $\mathbb{R}_+$-graded $\k[\mathbb{R}_+]$-modules and maps the graded module homomorphisms.

\begin{dfn} \label{dfn:simplicialmodule}
A {\bf $\mathbb{R}_+$-graded simplicial $\k[\mathbb{R}_+]$-module} $M$ is a simplicial object in $\k[\mathbb{R}_+]$-Mod, i.e. a contravariant functor $M : \ul{\Delta} \rightarrow \k[\mathbb{R}_+]$-Mod, where $\ul{\Delta}$ is the simplex category (see Section~\ref{sec:simplicialsets}).
\end{dfn}

Concretely, $M$ is a collection of $\mathbb{R}_+$-graded $\k[\mathbb{R}_+]$-modules $M_n = M([n])$ for $n \geq 0$ together with $\mathbb{R}_+$-graded $\k[\mathbb{R}_+]$-module maps $M_{\varphi} : M_n \rightarrow M_m$ for every map $\varphi : [m] \rightarrow [n]$ in $\ul{\Delta}$ that together satisfy the usual functorial properties. We call $M_n$ the \emph{$n$-simplices} of $M$. The images $M_{d^{\Delta}_i}$ and $M_{s^{\Delta}_i}$ of the coface maps $d^{\Delta}_i$ and the codegeneracy maps $s^{\Delta}_i$ are called face maps and degeneracy maps, respectively, and for convenience we often denote them by $d_i$ and $s_i$.

There is a natural chain complex associated to a simplicial module.

\begin{dfn} \label{dfn:alternatingfacemapscomplex}
The {\bf alternating face maps complex} $C_*(M)$ associated to a graded simplicial $\k[\mathbb{R}_+]$-module $M$ is a chain complex of $\mathbb{R}_+$-graded $\k[\mathbb{R}_+]$-modules which in level $n \geq 0$ is the module
$$C_n(M) = M_n$$
and where the differential is given by the alternating sum of the face maps, i.e.
\begin{align*}
d : C_{n+1}(M) &\rightarrow C_n(M)\\
x &\mapsto \sum_{i=0}^{n+1} (-1)^id_i(x).
\end{align*}
\end{dfn}

\subsection{The Algebraic Persistent K\"unneth Theorem} \label{sec:algebraickunneth}

For graded simplicial $\k[\mathbb{R}_+]$-modules $M$ and $N$ (Definition~\ref{dfn:simplicialmodule}), the tensor product $M \otimes N$ is defined level-wise, i.e.
$$(M \otimes N)_n := M_n \otimes_{\k[\mathbb{R}_+]} N_n$$
for $n \geq 0$. For a morphism $\varphi : [m] \rightarrow [n]$ in $\ul{\Delta}$, the corresponding map in $M \otimes N$ is given by
\begin{align*}
(M \otimes N)_{\varphi} : M_n \otimes N_n &\rightarrow M_m \otimes N_m\\
x \otimes y &\mapsto M_{\varphi}(x) \otimes N_{\varphi}(y).
\end{align*}
We establish in Theorem~\ref{thm:kunnethformodules} a K\"unneth formula for the homology of $C_*(M \otimes N)$ in terms of $C_*(M)$ and $C_*(N)$, where $C_*(-)$ is the alternating face maps complex (Definition~\ref{dfn:alternatingfacemapscomplex}).

In level $n \geq 0$, we have
$$C_n(M \otimes N) = C_n(M) \otimes_{\k[\mathbb{R}_+]} C_n(N).$$
There is another chain complex of $\k[\mathbb{R}_+]$-modules associated to the pair $M,N,$ given by the usual tensor product of chain complexes ${C_*(M) \otimes C_*(N)}$, which in level $n$ is the $\k[\mathbb{R}_+]$-module
$$(C_*(M) \otimes C_*(N))_n = \bigoplus_{i+j=n} C_i(M) \otimes_{\k[\mathbb{R}_+]} C_j(N).$$
The Eilenberg-Zilber theorem (see \cite{eilenberg1953products}, \cite[Thm.~8.5.1]{weibel1995introduction}) says that there is a natural isomorphism on homology
\begin{equation} \label{eq:eilenbergzilberiso}
H_*(C_*(M) \otimes C_*(N)) \cong H_*(C_*(M \otimes N)).
\end{equation}

\begin{dfn}
A simplicial module $M$ is {\bf graded free} (resp. {\bf finitely generated}) if, for all $n \geq 0$, the module $M_n$ is graded free (resp. finitely generated).
\end{dfn}

We now prove Theorem~\ref{thm:intro-kunnethformodules} from the introduction.

\begin{thm} {\bf (Algebraic Persistent K\"unneth Theorem)} \label{thm:kunnethformodules}
Let $M,N$ be $\mathbb{R}_+$-graded free finitely generated simplicial $\k[\mathbb{R}_+]$-modules. Then for $n \geq 0$ there is a short exact sequence
\begin{align*}
0 &\rightarrow \bigoplus_{i + j = n} H_i(C_*(M)) \otimes_{\k[\mathbb{R}_+]} H_j(C_*(N)) \rightarrow H_n(C_*(M \otimes_{\k[\mathbb{R}_+]} N))\\
&\rightarrow \bigoplus_{i+j = n-1} Tor_1(H_i(C_*(M)), H_j(C_*(N))) \rightarrow 0
\end{align*}
which is natural in both $M$ and $N$. Moreover, the sequence splits, but not naturally.
\end{thm}
\begin{proof}
For $n \geq 0$, the submodules $B_n(M)$ and $Z_n(M)$ of $C_n(M) = M_n$ consisting of boundaries and cycles, respectively, are graded free by Lemma~\ref{lem:finitelygeneratedfree}, and similarly for the boundaries and cycles in $C_n(N)$. Hence, although $\k[\mathbb{R}_+]$ is not a principal ideal domain, the proof of the  K\"unneth short exact sequence for the homology of the tensor product of chain complexes $C_*(M) \otimes C_*(N)$ still goes through; see for example the proof of \cite[Theorem~3B.5]{Hatcher}. That is, for $n \geq 0$ there is a natural short exact sequence
\begin{align*}
0 &\rightarrow \bigoplus_{i + j = n} H_i(C_*(M)) \otimes_{\k[\mathbb{R}_+]} H_j(C_*(N)) \rightarrow H_n(C_*(M) \otimes C_*(N))\\
&\rightarrow \bigoplus_{i+j = n-1} Tor_1(H_i(C_*(M)), H_j(C_*(N))) \rightarrow 0
\end{align*}
which splits (but the splitting is not natural). The theorem then follows from the natural Eilenberg-Zilber isomorphism \eqref{eq:eilenbergzilberiso}.
\end{proof}

We can compute $H_n(C_*(M \otimes_{\k[\mathbb{R}_+]} N))$ using the K\"unneth short exact sequence in Theorem~\ref{thm:kunnethformodules} together with the classification result Proposition~\ref{prp:classification} and the calculations of tensor product and Tor of bars in Proposition~\ref{prp:tensortorbars}.

\section{$\mathbb{R}_+$-filtered simplicial sets} \label{sec:filteredsimplicialsets}

We define a notion of a simplicial set filtered by $\mathbb{R}_+ = [0,\infty)$ (see Section~\ref{sec:simplicialsets} for a review of simplicial sets). Its homology is a $\mathbb{R}_+$-graded module over the ring $\k[\mathbb{R}_+]$, whose theory we discuss in Section~\ref{sec:modules}. In Section~\ref{sec:metricspaces} we associate a $\mathbb{R}_+$-filtered simplicial set to a metric space $(X,d)$ and we define persistent homology of the metric space $PH_*(X)$ to be the homology of the filtered simplicial set. This is a rephrasing of the Vietoris-Rips construction; see Remark~\ref{rmk:vietorisrips}.

In this section we prove the K\"unneth theorem for $\mathbb{R}_+$-filtered simplicial sets (Theorem~\ref{thm:kunnethfilteredsimplicialsets}), which follows easily from the algebraic K\"unneth theorem for simplicial $\k[\mathbb{R}_+]$-modules (Theorem~\ref{thm:kunnethformodules}).

\begin{dfn} \label{dfn:filteredsimplicialset}
A {\bf $\mathbb{R}_+$-filtered simplicial set} $({\bf X},l)$ is a simplicial set ${\bf X}$ together with maps $l : {\bf X}_n \rightarrow \mathbb{R}_+$ for all $n \geq 0$ satisfying $l({\bf X}_{\varphi}(\s)) \leq l(\s)$ for all simplices $\s \in {\bf X}_n$ and all morphisms $\varphi : [m] \rightarrow [n]$ in $\ul{\Delta}$.
\end{dfn}

\begin{rmk}
A natural notion of a $\mathbb{R}_+$-filtered simplicial set would be a collection of simplicial sets $\{{\bf X}^{(t)}\}_{t \in \mathbb{R}_+}$ and inclusion maps ${\bf X}^{(t)} \hookrightarrow {\bf X}^{(t')}$ for $t \leq t'$. Given a $\mathbb{R}_+$-filtered simplicial set $({\bf X},l)$ as in Definiton~\ref{dfn:filteredsimplicialset}, for any $t \in \mathbb{R}_+$, the collection of simplices $\s \in {\bf X}_*$ such that $l(\s) \leq t$ forms a simplicial set ${\bf X}^{(t)}$, and if $t \leq t'$ then there is an inclusion ${\bf X}^{(t)} \hookrightarrow {\bf X}^{(t')}$.

The persistent homology $PH_*({\bf X},l)$ (see Definition~\ref{dfn:persistenthomologyfilteredsimplicialset}) is a $\mathbb{R}_+$-graded $\k[\mathbb{R}_+]$-module with homogeneous degree-$t$ part equal to the homology of the simplicial set ${\bf X}^{(t)}$, i.e. $PH_*^{(t)}({\bf X},l) = H_*({\bf X}^{(t)};\k).$
\end{rmk}

Let $\mathbb{R}_+$-sSet denote the category\footnote{Alternatively, we may define $\mathbb{R}_+$-sSet as follows. Consider the category of pairs $(A,l)$ of sets $A$ and maps $l : A \rightarrow \mathbb{R}_+$, where the morphisms $f : (A,l) \rightarrow (A',l')$ are the set maps $f : A \rightarrow A'$ such that $l'(f(a)) \leq l(a)$ for all $a \in A$. Such a pair $(A,l)$ is called a $\mathbb{R}_+$-filtered set, as defined in \cite[Sec.~3.4]{carlsson2014topological}, and we denote the category of these pairs by $\mathbb{R}_+$-Set. Then a $\mathbb{R}_+$-filtered simplicial set is a simplicial object in $\mathbb{R}_+$-Set, i.e. a contravariant functor $\ul{\Delta} \rightarrow \mathbb{R}_+$-Set, and $\mathbb{R}_+$-sSet is the category of these functors.} with objects the $\mathbb{R}_+$-filtered simplicial sets $({\bf X},l)$ and the morphisms $f : ({\bf X},l) \rightarrow ({\bf X}',l')$ given by the simplicial maps ${\bf X} \rightarrow {\bf X'}$ satisfying $l'(f(\s)) \leq l(\s)$ for all $\s \in {\bf X}.$

Let $\k[\mathbb{R}_+]$-sMod denote the category of $\mathbb{R}_+$-graded simplicial  $\k[\mathbb{R}_+]$-modules (Definition~\ref{dfn:simplicialmodule}).

We now define a functor
\begin{equation} \label{eq:persistentfunctor}
F : \mathbb{R}_+\text{-sSet} \rightarrow \k[\mathbb{R}_+]\text{-sMod}
\end{equation}
and then define the persistent homology of $({\bf X},l)$ (Definition~\ref{dfn:persistenthomologyfilteredsimplicialset}) to be the homology of $F({\bf X},l)$. Let the $n$-simplices of $F({\bf X},l)$ be the free module
\begin{equation} \label{eq:persistentchainsset}
F_n({\bf X},l) := \bigoplus_{\s \in {\bf X}_n} \G^{l(\s)} \k[\mathbb{R}_+] \cdot  \langle \s \rangle,
\end{equation}
where the notation $\G^a$ means a degree shift of $a \in \mathbb{R}_+$. So, $F$ is essentially the free $\k[\mathbb{R}_+]$-module functor, with additional degree shifts of the generators by the filtration function $l$. For a morphism $\varphi : [m] \rightarrow [n]$ in $\ul{\Delta}$, we define the corresponding morphism in $F({\bf X},l)$ on generators by
\begin{align*}
F({\bf X},l)_{\varphi} : F_n({\bf X},l) &\rightarrow F_m({\bf X},l)\\
\langle \s \rangle &\mapsto T^{l(\s) - l({\bf X}_{\varphi}(\s))}\langle {\bf X}_{\varphi}(\s) \rangle.
\end{align*}
Extending $\k[\mathbb{R}_+]$-linearly defines $F({\bf X},l)_{\varphi}$ on all of $F_n({\bf X},l)$. Note that $F({\bf X},l)_{\varphi}$ is a graded map due to the degree shifts in \eqref{eq:persistentchainsset}. In particular, the face maps on $F({\bf X},l)$ are given on generators by
\begin{equation*} \label{eq:facemapsonF}
d_i(\langle \s \rangle) := T^{l(\s) - l(d_i(\s))}\langle d_i(\s) \rangle,
\end{equation*}
where on the left we are writing $d_i = F({\bf X},l)_{d_i^{\Delta}}$ and on the right we are writing $d_i = {\bf X}_{d_i^{\Delta}}.$

For a map $f : ({\bf X},l) \rightarrow ({\bf X}',l')$ in $\mathbb{R}_+$-sSet, we define $F(f)$ on generators by
\begin{align*}
F(f) : F_n({\bf X},l) &\rightarrow F_n({\bf X}',l')\\
\langle \s \rangle &\mapsto T^{l(\s) - l'(f(\s))}\langle f(\s) \rangle.
\end{align*}
Note that the property $l(\s) - l'(f(\s)) \geq 0$, which holds by definition of a map\\ $f \in \mathbb{R}_+$-sSet, is crucial in the definition of $F(f)$; indeed, the factor $T^{l(\s) - l'(f(\s))}$ is necessary to make $F(f)$ a $\mathbb{R}_+$-graded map due to the degree shifts in \eqref{eq:persistentchainsset}.

\begin{dfn} \label{dfn:persistenthomologyfilteredsimplicialset}
Given $({\bf X},l) \in \mathbb{R}_+$-sSet, its {\bf persistent chain complex} $PC_*({\bf X},l)$ is the alternating face maps complex (Defintion~\ref{dfn:alternatingfacemapscomplex}) of the module $F({\bf X},l)$,
$$PC_*({\bf X},l) := C_*(F({\bf X},l)),$$
and its {\bf persistent homology} is the homology $PH_*({\bf X},l) := H_*(PC_*({\bf X},l)).$ For $t \in \mathbb{R}_+$, the {\bf time-$t$ persistent homology} $PH_*^{(t)}({\bf X},l)$ is the homogeneous degree $t$ part of $PH_*({\bf X},l)$.
\end{dfn}

Consider $\mathbb{R}_+$-filtered simplicial sets $({\bf X},l_X)$ and $({\bf Y},l_Y)$. We define their Cartesian product by
$$({\bf X},l_X) \times ({\bf Y},l_Y) := ({\bf X} \times {\bf Y}, l_X+ l_Y),$$
where the simplicial set ${\bf X} \times {\bf Y}$ has $n$-simplices $({\bf X} \times {\bf Y})_n = {\bf X}_n \times {\bf Y}_n$ and structure maps $({\bf X} \times {\bf Y})_{\varphi} = {\bf X}_{\varphi} \times {\bf Y}_{\varphi}$ for $\varphi \in \ul{\Delta}$, and where $(l_X + l_Y)(\s,\t) = l_X(\s) + l_Y(\t).$

We now establish the K\"unneth formula (Theorem~\ref{thm:kunnethfilteredsimplicialsets}) that computes persistent homology of the Cartesian product $PH_*({\bf X} \times {\bf Y}, l_X \times l_Y)$ when ${\bf X}_n$ and ${\bf Y}_n$ are finite sets for $n \geq 0$.

\begin{proof}[Proof of Theorem~\ref{thm:kunnethfilteredsimplicialsets}]
Recall the functor $F : \mathbb{R}_+\text{-sSet} \rightarrow \k[\mathbb{R}_+]\text{-sMod}$ from \eqref{eq:persistentfunctor}. There is a canonical isomorphism
$$F({\bf X} \times {\bf Y}, l_X + l_Y) \cong F({\bf X},l_X) \otimes_{\k[\mathbb{R}_+]} F({\bf Y},l_Y).$$
Then the theorem follows from the definition of persistent homology (Definition~\ref{dfn:persistenthomologyfilteredsimplicialset}) and the K\"unneth formula in $\k[\mathbb{R}_+]$-sMod (Theorem~\ref{thm:kunnethformodules}).
\end{proof}

\section{Metric spaces} \label{sec:metricspaces}

The main purpose of this paper is to investigate the persistent homology of the sum metric $d_X + d_Y$ on the product $X \times Y$ of metric spaces $(X, d_X)$ and $(Y, d_Y)$.

In Section~\ref{sec:persistenthomologymetricspaces}, we define persistent homology $PH_*(X)$ of a metric space $(X, d_X)$ as the persistent homology $PH_*({\bf X},l_X)$ of a $\mathbb{R}_+$-filtered simplicial set $({\bf X}, l_X)$ (Definitions~\ref{dfn:filteredsimplicialset},\,\ref{dfn:persistenthomologyfilteredsimplicialset}) associated to $(X,d_X)$. See Remark~\ref{rmk:vietorisrips} for a discussion of the equivalence of this definition to the usual Vietoris-Rips construction.

The K\"unneth theorem for $\mathbb{R}_+$-filtered simplicial sets (Theorem~\ref{thm:kunnethfilteredsimplicialsets}) computes\\ $PH_*({\bf X} \times {\bf Y}, l_X + l_Y)$, however the filtration function $l_{X \times Y}$ associated to the sum metric $d_X + d_Y$ is not equal to $l_X + l_Y$; it only satisfies the inequality (see \eqref{eq:filtrationinequality})
$$l_{X \times Y} \leq l_X + l_Y.$$
We discuss this in Section~\ref{sec:les}. The inequality provides a graded module homomorphism
\begin{equation}\label{eq:comparisonmapmetricintro}
PH_*(X, Y) := PH_*({\bf X} \times {\bf Y}, l_X + l_Y) \rightarrow PH_*(X \times Y)
\end{equation}
which fits into a long exact sequence (Proposition~\ref{prp:les}). The third term in the triple that forms the long exact sequence is denoted $\overline{PH}_*(X,Y)$.

In Section~\ref{sec:kunnethinlowdims} we prove the vanishing $\overline{PH}_n(X,Y) = 0$ in dimensions $n = 0,1,2$. This implies that the mapping \eqref{eq:comparisonmapmetricintro} is an isomorphism in dimensions $0,1$ and a surjection in dimension $2$. Then from the K\"unneth Theorem~\ref{thm:kunnethfilteredsimplicialsets} for $PH_n({\bf X} \times {\bf Y}, l_X + l_Y)$, we obtain our K\"unneth results for $PH_n(X \times Y)$ in low dimensions $n = 0,1,2$ (Theorem~\ref{thm:kunnethmain}).

In Section~\ref{sec:interleavingdistance} we bound the interleaving distance between the computable module $PH_*(X,Y)$ and the homology $PH_*(X \times Y)$ that we are interested in  by the minimum of the diameters of $X$ and $Y$; see Theorem~\ref{thm:interleavingdistance}.

In Section~\ref{sec:hammingcube}, we use Theorem~\ref{thm:kunnethmain} to compute $PH_n$ of the Hamming cube $\{0,1\}^k$ for $n = 0,1,2$ and for all $k \geq 1$. This involves a technical computation of the nontrivial kernel of the surjection \eqref{eq:comparisonmapmetricintro} in dimension $2$  for $X = \{0,1\}$ and $Y = \{0,1\}^{k-1}$.

\subsection{Persistent homology} \label{sec:persistenthomologymetricspaces}

We associate a $\mathbb{R}_+$-filtered simplicial set (Definition~\ref{dfn:filteredsimplicialset}) to every metric space, and we define the persistent homology of the metric space (Definition~\ref{dfn:persistenthomologymetricspace}) to be the persistent homology of the associated filtered simplicial set (Definition~\ref{dfn:persistenthomologyfilteredsimplicialset}). This is a simplicial set version of the usual Vietoris-Rips construction; see Remark~\ref{rmk:vietorisrips} for further discussion.

Let $Met$ denote the category with metric spaces $(X,d_X)$ as objects and with morphisms the maps $f : (X,d_X) \rightarrow (Y,d_Y)$ that are distance non-increasing, i.e.
\begin{equation} \label{eq:distancenonincreasing}
d_Y(f(a),f(b)) \leq d_X(a,b) \text{ for all } a,b \in X.
\end{equation}
Recall from Section~\ref{sec:filteredsimplicialsets} the category $\mathbb{R}_+$-sSet of $\mathbb{R}_+$-filtered simplicial sets.

We now describe a functor
\begin{align*}
Met &\rightarrow \mathbb{R}_+\text{-sSet}\\
(X, d_X) &\mapsto ({\bf X}, l_X)
\end{align*}
and define persistent homology of the metric space $(X, d_X)$ (Definition~\ref{dfn:persistenthomologymetricspace}) to be the persistent homology of $({\bf X}, l_X)$.

Let $(X,d_X) \in Met$. There is a simplicial set ${\bf X}$ with $n$-simplices
$${\bf X}_n = \{ (x_0,\ldots,x_n) \,\, | \,\, x_i \in X \text{ for all } 0 \leq i \leq n\}$$
consisting of all ordered $(n+1)$-tuples of points in $X$. For a morphism $\varphi : [m] \rightarrow [n]$ in $\ul{\Delta}$, the map ${\bf X}_{\varphi} : {\bf X}_n \rightarrow {\bf X}_m$ is defined on a $n$-simplex $\s = (x_0,\ldots,x_n) \in {\bf X}_n$ by
$${\bf X}_{\varphi}(x_0,\ldots,x_n) = (x_{\varphi(0)},\ldots,x_{\varphi(m)}).$$
Note that, for the maps ${\bf X}_{\varphi}$ to be well-defined, it is essential that the simplices $\s \in {\bf X}_n$ are taken to be \emph{ordered} tuples.

There is also a \emph{max-length} map
\begin{align}
l_X : {\bf X}_n &\rightarrow \mathbb{R}_+ \label{eq:maxlength}\\
(x_0,\ldots,x_n) &\mapsto \max \{ d_X(x_i,x_j) \,\, | \,\, 0 \leq i,j \leq n \} \nonumber
\end{align}
given by the maximum of the pairwise distances between points in a simplex. We observe $l_X({\bf X}_{\varphi}(\s)) \leq l_X(\s)$ holds since the points in ${\bf X}_{\varphi}(\s)$ are a subset of the points in $\s$. Hence, the pair $({\bf X},l_X)$ is a $\mathbb{R}_+$-filtered simplicial set. This defines the functor on objects $(X, d_X) \mapsto ({\bf X}, l_X)$. For a morphism $f : (X,d_X) \rightarrow (Y,d_Y)$ in $Met$ there is an obvious induced map on simplicial sets ${\bf f} : {\bf X} \rightarrow {\bf Y}$ that satisfies $l_Y({\bf f}(\s)) \leq l_X(\s)$ by the distance non-increasing property \eqref{eq:distancenonincreasing} of $f$. So, ${\bf f}$ is indeed a morphism in $\mathbb{R}_+$-sSet. This completes the definition of the functor $Met \rightarrow \mathbb{R}_+\text{-sSet}$.

We now define the persistent homology of a metric space $(X,d_X)$.

\begin{dfn} \label{dfn:persistenthomologymetricspace}
Given a metric space $(X,d_X)$, its {\bf persistent chain complex} $PC_*(X)$ is the persistent chain complex of the associated $\mathbb{R}_+$-filtered simplicial set $({\bf X}, l_X)$ (see Definition~\ref{dfn:persistenthomologyfilteredsimplicialset}), i.e.
$$PC_*(X) := PC_*({\bf X},l_X).$$

The {\bf persistent homology} of $(X,d_X)$ is the homology $PH_*(X) := PH_*({\bf X},l_X)$ of this chain complex.

For $t \in \mathbb{R}_+$, the {\bf time-$t$ persistent homology} $PH_*^{(t)}(X)$ is the homogeneous degree $t$ part of $PH_*(X)$.
\end{dfn}

\begin{rmk} \label{rmk:vietorisrips} {\bf (Vietoris-Rips)}
The classical definition of persistent homology of a metric space $(X,d_X)$ uses the Vietoris-Rips construction, which we now recall. For $t \in \mathbb{R}_+$, the time-$t$ Vietoris-Rips complex of $X$ is the abstract simplicial complex $V(X,t)$ with $n$-simplices given by all subsets $\{x_0,\ldots,x_n\} \subset X$ satisfying $d_X(x_i,x_j) \leq t$ for all $0 \leq i,j \leq n$.  For $s \leq t$ there is an inclusion $V(X,s) \hookrightarrow V(X,t)$ which induces a map on homology $H_*(V(X,s)) \rightarrow H_*(V(X,t))$. The homology of Vietoris-Rips at all times $t \in \mathbb{R}_+$ together with the induced maps of the inclusions for $s \leq t$ forms the persistent homology of $X$.

Letting $({\bf X}, l_X)$ denote the $\mathbb{R}_+$-filtered simplicial set associated to $(X,d_X)$, for all $t \in \mathbb{R}_+$ there is a simplicial set
$${\bf X}^{(t)} := \{ \text{simplices } \s \text{ in } {\bf X} \,\, | \,\, l_X(\s) \leq t\}.$$

There is a homotopy equivalence $V(X,t) \rightarrow {\bf X}^{(t)}.$  The homogeneous degree-$t$ persistent homology of $(X,d_X)$ is the homology of ${\bf X}^{(t)}.$ In conclusion, for all $t \in \mathbb{R}_+$ there are isomorphisms of $\k$-vector spaces
$$PH^{(t)}_*(X) \cong H_*({\bf X}^{(t)}) \cong H_*(V(X,t))$$
which are natural in $t$, i.e. the map $H_*(V(X,s)) \rightarrow H_*(V(X,t))$ for $s \leq t$ is identified with the map $PH^{(s)}_*(X) \rightarrow PH^{(t)}_*(X)$ given by multiplication by $T^{t-s} \in \k[\mathbb{R}_+]$.
\end{rmk}

\subsection{A long exact sequence for the sum metric $d_X + d_Y$} \label{sec:les}

Consider metric spaces $X$ and $Y$ with metrics $d_X$ and $d_Y$. Ultimately, we are interested in a K\"unneth formula for $PH_*(X \times Y)$ in terms of $PH_*(X)$ and $PH_*(Y)$, where the product space $X \times Y$ is equipped with the sum metric $d_{X \times Y} = d_X + d_Y$.

In this section, we describe a long exact sequence in Proposition~\ref{prp:les} that relates the persistent homology of the product metric space $(X \times Y,d_X + d_Y)$ and the persistent homology
\begin{equation} \label{eq:predictedmodule}
PH_*(X, Y) := PH_*({\bf X} \times {\bf Y}, l_X + l_Y)
\end{equation}
of the product filtered simplicial set ${\bf X} \times {\bf Y}$ with filtration function $l_X + l_Y$, where $l_X$ and $l_Y$ are the max-length maps on the simplicial sets ${\bf X}$ and ${\bf Y}$, respectively; see Section~\ref{sec:persistenthomologymetricspaces}. This is useful because the K\"unneth theorem for filtered simplicial sets (Theorem~\ref{thm:kunnethfilteredsimplicialsets}) provides the following formula for $PH_*(X, Y)$ in terms of $PH_*(X)$ and $PH_*(Y)$.

\begin{prp} \label{prp:kunnethdiagonal}
Given finite metric spaces $(X,d_X)$ and $(Y,d_Y)$, for $n \geq 0$ there is a short exact sequence
\begin{align} \label{eq:classicalkunnethses}
0 &\rightarrow \bigoplus_{i+j = n} PH_i(X) \otimes_{\k[\mathbb{R}_+]} PH_j(Y) \rightarrow PH_n(X, Y)\\
&\rightarrow \bigoplus_{i+j =n-1} Tor_1(PH_i(X),PH_j(Y)) \rightarrow 0 \nonumber
\end{align}
which is natural with respect to distance non-increasing maps $(X,d_X) \rightarrow (X',d_{X'})$ and $(Y,d_Y) \rightarrow (Y',d_{Y'})$. Moreover, the sequence splits, but not naturally.
\end{prp}
\begin{proof}
By definition of persistent homology of a metric space (Definition~\ref{dfn:persistenthomologymetricspace}) we have $PH_*(X) = PH_*({\bf X},l_X)$ and $PH_*(Y) = PH_*({\bf Y},l_Y)$. Hence the proposition follows immediately from Theorem~\ref{thm:kunnethfilteredsimplicialsets} and the definition \eqref{eq:predictedmodule}.
\end{proof}

We now describe the long exact sequence which relates $PH_*(X \times Y)$ to $PH_*(X,Y)$. By definition, we have the chain complexes
\begin{align} 
PC_n(X, Y) &:= PC_n({\bf X} \times {\bf Y}, l_X + l_Y) \label{eq:sumofmaxeschains}\\
&= \bigoplus_{(\s,\t) \in {\bf X}_n \times {\bf Y}_n} \G^{l_X(\s) + l_Y(\t)} \k[\mathbb{R}_+] \cdot \langle \s,\t \rangle \nonumber
\end{align}
and
\begin{equation} \label{eq:maxofsumschains}
PC_n(X \times Y) = \bigoplus_{(\s,\t) \in {\bf X}_n \times {\bf Y}_n} \G^{l_{X \times Y}(\s,\t)} \k[\mathbb{R}_+] \cdot \langle \s,\t \rangle.
\end{equation}
Given simplices $\s = (x_0,\ldots,x_n) \in {\bf X}_n$ and $\t = (y_0,\ldots,y_n) \in {\bf Y}_n$, we claim that the max-length $l_{X \times Y}(\s,\t)$ of the $n$-simplex $(\s,\t)$ in $X \times Y$ must be less than or equal to the sum $l_X(\s) + l_Y(\t)$. Indeed, for some $0 \leq i \leq j \leq n$, the max-length of $(\s,\t)$ is equal to $l_{X \times Y}(\s,\t) = d_{X \times Y}((x_i,y_i),(x_j,y_j)) = d_X(x_i,x_j) + d_Y(y_i,y_j)$. Then since $d_X(x_i,x_j) \leq l_X(\s)$ and $d_Y(y_i,y_j) \leq l_Y(\t)$, the claim follows:
\begin{equation} \label{eq:filtrationinequality}
l_{X \times Y}(\s,\t) \leq l_X(\s) + l_Y(\t) \text{ for all } (\s,\t) \in {\bf X}_n \times {\bf Y}_n.
\end{equation}
So, due to the degree shifts $\G$ in \eqref{eq:sumofmaxeschains} and \eqref{eq:maxofsumschains}, for $n \geq 0$ there is a $\mathbb{R}_+$-graded $\k[\mathbb{R}_+]$-linear embedding
\begin{align} \label{eq:comparisonmap}
\iota_n : PC_n(X, Y) &\hookrightarrow PC_n(X \times Y)\\
\langle \s,\t \rangle &\mapsto T^{l_X(\s) + l_Y(\t) - l_{X \times Y}(\s,\t)} \cdot \langle \s,\t \rangle. \nonumber
\end{align}
Together these form a chain map $\i : PC_*(X,Y) \hookrightarrow PC_*(X \times Y)$ and hence induce maps on homology
\begin{equation} \label{eq:inducedcomparisonmap}
{\iota_n}_* : PH_n(X, Y) \rightarrow PH_n(X \times Y).
\end{equation}
Consider the relative chain complex
\begin{align} \label{eq:relativechaincomplex}
\overline{PC}_n(X, Y) &:= PC_n(X \times Y)/\iota_n(PC_n(X, Y))\\
&= \bigoplus_{(\s,\t) \in {\bf X}_n \times {\bf Y}_n} \G^{l_{X \times Y}(\s,\t)} \k[\mathbb{R}_+] / (T^{l_X(\s) + l_Y(\t) - l_{X \times Y}(\s,\t)}) \cdot \langle \s,\t \rangle \nonumber
\end{align}
with homology $\overline{PH}_*(X, Y).$

Let $q_n : PC_n(X \times Y) \rightarrow \overline{PC}_n(X, Y)$ denote the quotient map implicit in the definition \eqref{eq:relativechaincomplex} and ${q_n}_*$ the induced map on homology. There is a connecting homomorphism $\delta_n : \overline{PH}_n(X, Y) \rightarrow PH_{n-1}(X, Y)$ defined in the usual way, which we now recall. Consider a chain $\alpha \in PC_n(X \times Y)$ that represents a cycle in $\overline{PC}_n(X, Y)$ and hence a homology class $[\alpha] \in \overline{PH}_n(X, Y).$ This means that $d(\alpha) \in \i_{n-1}(PC_{n-1}(X, Y)).$ Then define $\delta_n([\alpha]) := [\i_{n-1}^{-1}(d(\alpha))] \in PH_{n-1}(X, Y).$ It is routine to check that $\delta_n$ is independent of the choice of representative $\alpha$ of the class $[\alpha]$, and that it fits into the following long exact sequence. Moreover, this sequence is natural in the arguments $X$ and $Y$. The result is the following.

\begin{prp} \label{prp:les}
For metric spaces $(X,d_X)$ and $(Y,d_Y)$, there is a long exact sequence
$$
\cdots \xrightarrow{\delta_{n+1}} PH_n(X, Y) \xrightarrow{{\i_n}_*} PH_n(X \times Y) \xrightarrow{{q_n}_*} \overline{PH}_n(X, Y) \xrightarrow{\delta_n} \cdots
$$
which is natural with respect to distance non-increasing maps $(X,d_X) \rightarrow (X',d_{X'})$ and $(Y,d_Y) \rightarrow (Y',d_{Y'})$.
\hfill$\square$
\end{prp}

\subsection{The K\"unneth Theorem for $d_X + d_Y$ in low dimensions} \label{sec:kunnethinlowdims}

The purpose of this section is to prove our main Theorem~\ref{thm:kunnethmain}, which is a refined version of Theorem~\ref{thm:intro-maintheorem} from the introduction.

\begin{thm} \label{thm:kunnethmain} Consider metric spaces $(X,d_X)$ and $(Y,d_Y)$ and the product space equipped with the sum metric $(X \times Y, d_X + d_Y)$.

Then, the map ${\i_n}_* : PH_n(X,Y) \rightarrow PH_n(X \times Y)$ (see \eqref{eq:inducedcomparisonmap}) is an isomorphism for $n = 0,1,$ and is a surjection for $n = 2$.

In particular, by Proposition~\ref{prp:kunnethdiagonal}, if $X$ and $Y$ are finite sets, then for $n = 0,1$ there is a short exact sequence
\begin{align} \label{eq:persistentkunnethdim01ses}
0 &\rightarrow \bigoplus_{i+j = n} PH_i(X) \otimes_{\k[\mathbb{R}_+]} PH_j(Y) \rightarrow PH_n(X \times Y)\\
&\rightarrow \bigoplus_{i+j =n-1} Tor_1(PH_i(X),PH_j(Y)) \rightarrow 0 \nonumber
\end{align}
which is natural with respect to distance non-increasing maps $(X,d_X) \rightarrow (X',d_{X'})$ and $(Y,d_Y) \rightarrow (Y',d_{Y'})$. Moreover, this sequence splits, but not naturally.
\hfill$\square$
\end{thm}
\begin{proof}
In Lemma~\ref{lem:relative} below we prove that $\overline{PH}_n(X,Y) = 0$ for $n = 0,1,2$. Hence it follows from the long exact sequence in Proposition~\ref{prp:les} that ${\iota_n}_*$ is an isomorphism for $n = 0,1,$ and a surjection for $n = 2$, as claimed. The short exact sequence for $n = 0,1$ follows by replacing $PH_n(X,Y)$ with $PH_n(X \times Y)$ in the short exact sequence \eqref{eq:classicalkunnethses} from Proposition~\ref{prp:kunnethdiagonal}.
\end{proof}

The following lemma is used in the proof of Theorem~\ref{thm:kunnethmain} above.

\begin{lem} \label{lem:relative}
For $n = 0,1,2,$ the homology of the chain complex \eqref{eq:relativechaincomplex} vanishes $\overline{PH}_n(X, Y) = 0.$

In particular, the map ${\iota_n}_* : PH_n(X, Y) \rightarrow PH_n(X \times Y)$ (see \eqref{eq:inducedcomparisonmap}) is an isomorphism for $n = 0,1,$ and is a surjection for $n = 2$.
\end{lem}
\begin{proof}
In dimensions $n = 0$ and $n = 1$, we claim that every simplex $(\s,\t) \in {\bf X}_n \times {\bf Y}_n$ has the property $l_X(\s) + l_Y(\t) = l_{X \times Y}(\s,\t)$. In dimension $0$ this is because the max-length of any $0$-simplex is zero. In dimension $1$ this is because the max-length is the length of the unique edge in the simplex. Hence, by the definition \eqref{eq:relativechaincomplex} of the relative chain complex, we have $\overline{PC}_n(X, Y) = 0$ and hence $\overline{PH}_n(X, Y) = 0$ for $n = 0,1$.

We now consider dimension $n = 2$. Since $\overline{PC}_1(X,Y) = 0,$ we have
$$\overline{PH}_2(X, Y) = \overline{PC}_2(X, Y) \, / \, im(\overline{d} : \overline{PC}_3(X, Y) \rightarrow \overline{PC}_2(X, Y)).$$
So, we must show that the differential
$$\overline{d} : \overline{PC}_3(X, Y) \rightarrow \overline{PC}_2(X, Y)$$
is surjective. Let $(\s,\t) \in {\bf X}_2 \times {\bf Y}_2$ and consider the corresponding generator $[\langle \s,\t \rangle] \in \overline{PC}_2(X, Y)$. To show that $[\langle \s,\t \rangle]$ is in the image of $\overline{d}$ we must find some $\alpha \in PC_3(X \times Y)$ such that $d(\alpha) = \langle \s,\t \rangle + r$ for some $r \in \i_2(PC_2(X, Y)).$ This is possible by Lemma~\ref{lem:surjective} below. Hence $\overline{d}$ is surjective and so $\overline{PH}_2(X,Y) = 0$.

The final statement of the lemma about ${\i_n}_*$ for $n = 0,1,2$ is an immediate consequence of exactness of the long exact sequence in Proposition~\ref{prp:les} and the vanishing $\overline{PH}_n(X, Y) = 0 \text{ for } n = 0,1,2.$
\end{proof}

The following lemma is used in the proof of Lemma~\ref{lem:relative} above.

\begin{lem} \label{lem:surjective}
Given a $2$-simplex $(\s,\t) \in {\bf X}_2 \times {\bf Y}_2$, there exists an element $\alpha \in PC_3(X \times Y)$ such that $d(\alpha) = \langle \s,\t \rangle + r$ for some $r \in \i_2(PC_2(X, Y)).$
\end{lem}
\begin{proof}
Say $\s = (x_0, x_1, x_2)$ and $\t = (y_0, y_1, y_2)$. Up to permutations of the indices and switching the roles of $x$ and $y$, there are four possible cases:

\begin{enumerate}
\item $x_0 = x_1 = x_2$,\\
\item $x_0 = x_1 \neq x_2$ and $y_1 = y_2$,\\
\item $x_0 = x_1 \neq x_2$ and $y_1 \neq y_2$,\\
\item the $x_i$ are all distinct.\\
\end{enumerate}

Notice that if $\langle \s, \t \rangle \in \i_2(PC_2(X, Y))$, then $\alpha := 0$ and $r := -\langle \s, \t \rangle$ satisfies the required conditions. We claim that Case~$(i)$ and Case~$(ii)$ are of this form.

Indeed, in Case~$(i)$ we have $l_X(\s) = 0$, and moreover $l_X(\s) + l_Y(\t) = l_Y(\t) = l_{X \times Y}(\s,\t)$. Hence by definition \eqref{eq:comparisonmap} of $\i_2$, we have $\langle \s, \t \rangle = \iota_2(\langle \s, \t \rangle) \in \i_2(PC_2(X, Y))$, as claimed.

Similarly, in Case~$(ii)$, observe that $l_{X \times Y}(\s,\t) = d_{X \times Y}((x_0,y_0),(x_2,y_2)) = d_X(x_0,x_2) + d_Y(y_0,y_2) = l_X(\s) + l_Y(\t)$, and so $\langle \s, \t \rangle = \iota_2(\langle \s, \t \rangle) \in \i_2(PC_2(X, Y))$.

We now consider Case~$(iii)$. Consider the $3$-simplices
\begin{align*}
\s' := (x_0, x_0, x_1, x_2) \in {\bf X}_3,\\
\t' := (y_2, y_0, y_1, y_2) \in {\bf Y}_3.
\end{align*}
We claim that $\alpha := \langle \s', \t' \rangle \in PC_3(X \times Y)$ satisfies the required conditions. The faces of the $3$-simplex $(\s',\t') \in {\bf X}_3 \times {\bf Y}_3$ are
\begin{align*}
f_0 := d_0(\s',\t') &= (\s,\t),\\
f_1 := d_1(\s',\t') &= \{(x_0,y_2),(x_1,y_1),(x_2,y_2)\},\\
f_2 := d_2(\s',\t') &= \{(x_0,y_2),(x_0,y_0),(x_2,y_2)\},\\
f_3 := d_3(\s',\t') &= \{(x_0,y_2),(x_0,y_0),(x_1,y_1)\}.
\end{align*}
From the assumption $x_0 = x_1$, it follows that $l_{X \times Y}(\s,\t) = l_{X \times Y}(\s',\t')$. Hence by definition of the differential $d$ on $PC_*(X \times Y)$ we have
\begin{equation} \label{eq:caseiii}
d(\langle \s', \t' \rangle) = \langle \s, \t \rangle + \sum_{i=1}^3 (-1)^i \cdot T^{l_{X \times Y}(\s,\t) - l_{X \times Y}(f_i)} \cdot \langle f_i \rangle.
\end{equation}
Observe that, since $x_0 = x_1$, up to a permutation of the vertices Case~$(ii)$ applies to the simplices $f_1$ and $f_2$ and Case~$(i)$ applies to the simplex $f_3$. Hence for $i = 1,2,3,$ we have shown above that $\langle f_i \rangle = \i_2(\langle f_i \rangle).$ Hence the summation term in \eqref{eq:caseiii} is an element of $i_2(PC_2(X, Y))$, i.e. $d(\langle \s', \t' \rangle) = \langle \s, \t \rangle + r$ where $r \in i_2(PC_2(X, Y))$. This completes the proof of Case~$(iii)$.

We now consider Case~$(iv)$. By definition of the max-length $l_{X \times Y}(\s,\t)$ and the sum metric $d_{X \times Y} = d_X + d_Y$, we have in particular
\begin{align*}
d_X(x_0,x_2) + d_Y(y_0,y_2) &= d_{X \times Y}((x_0,y_0),(x_2,y_2)) \leq l_{X \times Y}(\s,\t),\\
d_X(x_1,x_2) + d_Y(y_1,y_2) &= d_{X \times Y}((x_1,y_1),(x_2,y_2)) \leq l_{X \times Y}(\s,\t).
\end{align*}
It follows that either
\begin{equation} \label{eqn:newside}
d(x_0,x_2) + d(y_1,y_2) \leq l_{X \times Y}(\s,\t)
\end{equation}
or
$$d(x_1,x_2) + d(y_0,y_2) \leq l_{X \times Y}(\s,\t)$$
holds. Without loss of generality, assume \eqref{eqn:newside} holds.

Consider the $3$-simplices
\begin{align*}
\s' := (x_0, x_0, x_1, x_2) \in {\bf X}_3,\\
\t' := (y_1, y_0, y_1, y_2) \in {\bf Y}_3.
\end{align*}
The faces of the $3$-simplex $(\s',\t') \in {\bf X}_3 \times {\bf Y}_3$ are
\begin{align*}
f_0 := d_0(\s',\t') &= (\s,\t),\\
f_1 := d_1(\s',\t') &= \{(x_0,y_1),(x_1,y_1),(x_2,y_2)\},\\
f_2 := d_2(\s',\t') &= \{(x_0,y_1),(x_0,y_0),(x_2,y_2)\},\\
f_3 := d_3(\s',\t') &= \{(x_0,y_1),(x_0,y_0),(x_1,y_1)\}.
\end{align*}
By \eqref{eqn:newside}, an examination of all edge lengths shows that
$$l_{X \times Y}(\s',\t') = l_{X \times Y}(\s,\t)$$
holds. Hence by definition of the differential $d$ on $PC_*(X \times Y)$ we have
$$d(\langle \s', \t' \rangle) = \langle \s, \t \rangle + \sum_{i=1}^3 (-1)^i \cdot T^{l_{X \times Y}(\s,\t) - l_{X \times Y}(f_i)} \cdot \langle f_i \rangle.$$
Observe that, up to a permutation of the vertices and interchanging the roles of $x$ and $y$, the simplices $f_1,f_2$ and $f_3$ are all of type $(i),(ii),$ or $(iii)$. So, the lemma says that for $i = 1,2,3,$ there exists $\alpha_i \in PC_3(X \times Y)$ such that $d(\alpha_i) = \langle f_i \rangle + r_i$ for some $r_i \in \i_2(PC_2(X, Y)).$ We claim that
$$\alpha := \langle \s', \t' \rangle + \sum_{i=1}^3 (-1)^{i+1} \cdot T^{l_{X \times Y}(\s,\t) - l_{X \times Y}(f_i)} \cdot \alpha_i$$
satisfies the required conditions; indeed,
\begin{align*}
d(\alpha) &= \langle \s, \t \rangle + \sum_{i=1}^3 \bigg [ (-1)^i \cdot T^{l_{X \times Y}(\s,\t) - l_{X \times Y}(f_i)} \cdot \langle f_i \rangle\\
&\,\,\,\,\,\,\,\,\,\,\,\,+ (-1)^{i+1} \cdot T^{l_{X \times Y}(\s,\t) - l_{X \times Y}(f_i)} \cdot d(\alpha_i) \bigg ]\\
&= \langle \s, \t \rangle + \sum_{i=1}^3 (-1)^i \cdot T^{l_{X \times Y}(\s,\t) - l_{X \times Y}(f_i)} \cdot ( \langle f_i \rangle - d(\alpha_i))\\
&= \langle \s, \t \rangle + \sum_{i=1}^3 (-1)^i \cdot T^{l_{X \times Y}(\s,\t) - l_{X \times Y}(f_i)} \cdot ( - \, r_i).
\end{align*}
This completes the proof of Case~$(iv)$, and hence the lemma.
\end{proof}

\subsection{Interleaving distance} \label{sec:interleavingdistance}
We bound the interleaving distance between the module $PH_*(X,Y)$ defined in \eqref{eq:predictedmodule}, which we can compute using the split short exact sequence in Proposition~\ref{prp:kunnethdiagonal}, and the homology $PH_*(X \times Y)$ of the metric space $(X \times Y, d_X + d_Y)$ that we would like to compute; see Theorem~\ref{thm:interleavingdistance}. Leonid Polterovich suggested we look for a bound of this type.

Here we view $\k[\mathbb{R}_+]$-modules as persistence vector spaces $(t \mapsto PH^{(t)}_*(X,Y))$ and $(t \mapsto PH^{(t)}_*(X \times Y))$; see Remark~\ref{rmk:artinrees}. The interleaving distance is a quantitative measure of how far away the computable module $PH_*(X,Y)$ is from the true persistent homology $PH_*(X \times Y)$ of the product $(X \times Y, d_X + d_Y)$. The long exact sequence in Proposition~\ref{prp:les} can be viewed as an algebraic measure of the difference between these modules.

The general definition of interleaving distance between persistence vector spaces is as follows. The interleaving distance was introduced in \cite{chazal2009proximity}, and has since become a standard measure of distance between persistence vector spaces.
\begin{dfn} \label{dfn:interleavingdistance}
Consider persistence vector spaces $V = \{V_t\}_{t \in \mathbb{R}_+}$ and $W = \{W_t\}_{t \in \mathbb{R}_+}$ with structure maps $l_{t,t'} : V_t \rightarrow V_{t'}$ and $s_{t,t'} : W_t \rightarrow W_{t'}$, respectively, for $t \leq t'$.

For $\delta \geq 0$, we say that $V$ and $W$ are {\bf $\delta$-interleaved} if there exist linear maps
\begin{align*}
F_t &: V_t \rightarrow W_{t + \delta},\\
G_t &: W_t \rightarrow V_{t + \delta}
\end{align*}
such that, for all $t \leq t'$, we have
\begin{align*}
l_{t+\delta,t'+\delta} \circ G_t = G_{t'} \circ s_{t,t'},\,\,\,\,
s_{t+\delta,t'+\delta} \circ F_t = F_{t'} \circ l_{t,t'},
\end{align*}
and
\begin{align*}
G_{t+\delta} \circ F_t = l_{t,t+2\delta},\,\,\,\, F_{t+\delta} \circ G_t = s_{t,t+2\delta}.
\end{align*}

The {\bf interleaving distance} between $V$ and $W$ is the infimum over all $\delta \geq 0$ such that $V$ and $W$ are $\delta$-interleaved.
\end{dfn}

\begin{thm} \label{thm:interleavingdistance}
The interleaving distance between the persistence vector spaces $(t \mapsto PH_*^{(t)}(X,Y))$ and $(t \mapsto PH_*^{(t)}(X \times Y))$ is less than or equal to the minimum of the diameters
$$\min(\text{\rm diameter}(X),\text{\rm diameter}(Y)),$$
where the diameter is the maximum distance
$$\text{diameter}(X) = \max \{d_X(x_0,x_1) \,\, | \,\, x_0,x_1 \in X\}.$$
\end{thm}
\begin{proof}
Let $\delta := \min(\text{\rm diameter}(X),\text{\rm diameter}(Y))$ denote the desired bound on the interleaving distance. We must show that the persistence vector spaces $(t \mapsto PH_*^{(t)}(X,Y))$ and $(t \mapsto PH_*^{(t)}(X \times Y))$ are $\delta$-interleaved. Denote the structure maps on these persistence vector spaces by $\overline{l}_{t,t'} : PH_*^{(t)}(X,Y) \rightarrow PH_*^{(t')}(X,Y)$ and $\overline{s}_{t,t'} : PH_*^{(t)}(X \times Y) \rightarrow PH_*^{(t')}(X \times Y)$ for $t \leq t'$. Note that the structure maps $\overline{l}_{t,t'}$ and $\overline{s}_{t,t'}$ on homology are induced by the inclusions of chain complexes $l_{t,t'} : PC_*^{(t)}(X,Y) \rightarrow PC_*^{(t')}(X,Y)$ and $s_{t,t'} : PC_*^{(t)}(X \times Y) \rightarrow PC_*^{(t')}(X \times Y)$.

We claim it suffices to show that the persistence vector spaces $(t \mapsto PC_*^{(t)}(X,Y))$ and $(t \mapsto PC_*^{(t)}(X \times Y))$ are $\delta$-interleaved with respect to chain maps
\begin{align*}
F^{(t)} &: PC_*^{(t)}(X,Y) \rightarrow PC_*^{(t+\delta)}(X \times Y)\\
G^{(t)} &: PC_*^{(t)}(X \times Y) \rightarrow PC_*^{(t+\delta)}(X,Y).
\end{align*}
(The structure maps on these persistence vector spaces are the inclusions $l_{t,t'}$ and $s_{t,t'}$.) Indeed, since $F^{(t)}$ and $G^{(t)}$ are chain maps, they induce maps on homology $\overline{F}^{(t)} : PH_*^{(t)}(X,Y) \rightarrow PH_*^{(t+\delta)}(X \times Y)$ and $\overline{G}^{(t)} : PH_*^{(t)}(X \times Y) \rightarrow PH_*^{(t+\delta)}(X,Y)$, and so the required properties of a $\delta$-interleaving
$$\overline{G}^{(t+\delta)} \circ \overline{F}^{(t)} = \overline{l}_{t,t+2\delta}, \,\,\,\, \overline{F}^{(t+\delta)} \circ \overline{G}^{(t)} = \overline{s}_{t,t+2\delta},$$
$$\overline{l}_{t+\delta,t'+\delta} \circ \overline{G}^{(t)} = \overline{G}^{(t')} \circ \overline{s}_{t,t'},\,\,\,\,
\overline{s}_{t+\delta,t'+\delta} \circ \overline{F}^{(t)} = \overline{F}^{(t')} \circ \overline{l}_{t,t'},$$
follow from the corresponding properties on the chain level.

We now define the maps $F^{(t)}$ and $G^{(t)}$ and show that they form a $\delta$-interleaving on the chain level. By definition we have
\begin{align*}
PC_*^{(t)}(X,Y) = \bigoplus_{l_X(\s) + l_Y(\t) \leq t} \k \cdot \langle \s,\t \rangle,\\
PC_*^{(t)}(X \times Y) = \bigoplus_{l_{X \times Y}(\s,\t) \leq t} \k \cdot \langle \s,\t \rangle.
\end{align*}
Since we have $l_{X \times Y} \leq l_X + l_Y$ (see \eqref{eq:filtrationinequality}), it is immediate that $l_{X \times Y} \leq l_X + l_Y + \delta$ holds, and so we can define $F^{(t)}$ to be the natural inclusion that takes generators to generators $F^{(t)}(\langle \s,\t \rangle) = \langle \s,\t \rangle$.

Towards defining $G^{(t)}$, we first establish the bound
\begin{equation*} \label{eq:theotherbound}
l_{X \times Y}(\s,\t) \geq \max(l_X(\s), l_Y(\t)).
\end{equation*}
Given simplices $\s = (x_0,\ldots,x_n) \in {\bf X}_n$ and $\t = (y_0,\ldots,y_n) \in {\bf Y}_n$, there are integers $i,j \in \{0,\ldots,n\}$ such that $l_X(\s) = d_X(x_i,x_j)$, and similarly there are $i',j'$ such that $l_Y(\t) = d_X(y_{i'},y_{j'})$, by definition \eqref{eq:maxlength} of the filtration functions $l_X$ and $l_Y$. Hence by definition \eqref{eq:maxlength} of $l_{X \times Y}$ and the sum metric $d_{X \times Y} = d_X + d_Y$, we have
\begin{align*}
l_{X \times Y}(\s,\t) &= \max\{ d_{X \times Y}((x_a,y_a),(x_b,y_b)) \,\, | \,\, 0 \leq a,b \leq n \}\\
&= \max\{ d_{X}(x_a,x_b) + d_{Y}(y_a,y_b)  \,\, | \,\, 0 \leq a,b \leq n \}\\
&\geq \max\{d_X(x_i,x_j), d_Y(y_{i'},y_{j'})\}\\
&= \max(l_X(\s), l_Y(\t)),
\end{align*}
as claimed. From this, we derive the bound
\begin{align*}
l_X(\s) + l_Y(\t) - l_{X \times Y}(\s,\t) &\leq l_X(\s) + l_Y(\t) - \max(l_X(\s), l_Y(\t))\\
& = \min(l_X(\s), l_Y(\t))\\
&\leq \delta.
\end{align*}
Hence we have $l_X + l_Y \leq l_{X \times Y} + \delta$, and so we can define $G^{(t)}$ to be the natural inclusion $G^{(t)}(\langle \s,\t \rangle) = \langle \s,\t \rangle$.

It is immediate that the required properties of a $\delta$-interleaving hold, since the structure maps $l_{t,t'}$ and $s_{t,t'}$ as well as the interleaving maps $G^{(t)}$ and $F^{(t)}$ are all the natural inclusions of chain complexes defined on generators by $\langle \s,\t \rangle \mapsto \langle \s, \t \rangle.$ This completes the proof.
\end{proof}

\subsection{Example: The Hamming Cube $I^k$} \label{sec:hammingcube}

Let $I = \{0, 1\}$ denote the metric space consisting of $2$ points at distance $1$. Then for any integer $k \geq 1$, the Cartesian product $I^k$ with the sum metric is the \emph{Hamming $k$-cube}, consisting of all $k$-tuples of zeros and ones with distance given by the number of coordinates in which tuples differ.

The purpose of this section is to make the following computation using Theorem~\ref{thm:kunnethmain}. In particular, we carry out a technical computation of the nontrivial kernel of the surjection ${\i_2}_* : PH_2(X,Y) \rightarrow PH_2(X \times Y)$ in the case $X = I$ and $Y = I^{k-1},\, k \geq 3$.

\begin{prp} \label{prp:cube} For $k \geq 1$, the Hamming cube $I^k = \{0,1\}^k$ satisfies the following:
\begin{align*}
PH_0(I^k) &\cong (0,1)^{2^k - 1} \oplus (0,\infty),\\
PH_1(I^k) &\cong (1,2)^{k \cdot 2^{k-1} - (2^k - 1)},\\
PH_2(I^k) &= 0. \label{eq:ph2Ik}
\end{align*}
\end{prp}
\begin{proof}
See Proposition~\ref{prp:ph1lemma} and Proposition~\ref{prp:ph2lemma}.
\end{proof}

We introduce the following useful terminology for a standard coordinate embedding $I^r \hookrightarrow I^k$.

\begin{dfn} \label{dfn:coordinateinclusions}
Given $0 \leq r \leq k$, a point $\xi \in I^{k-r}$, and any choice of $r$ coordinates $1 \leq i_1 < i_2 < \cdots < i_r \leq k$ of $I^k$, there is an isometric embedding 
$$\varphi : I^r \rightarrow I^k,$$
where for $x \in I^r$ the $r$ chosen coordinates of $\varphi(x)$ are equal to $x$, i.e.
$$x = (x_1,\ldots,x_r) = (\varphi(x)_{i_1},\ldots, \varphi(x)_{i_r}),$$
and the other $k-r$ coordinates of $\varphi(x)$ are given by $\xi$. We call these maps $\varphi$ the {\bf coordinate inclusions of $I^r$ into $I^k$.}
\end{dfn}

For $r \leq k$, let $C(r,k)$ denote the set of coordinate inclusions $I^r \rightarrow I^k$.

\begin{prp} \label{prp:ph1lemma}
For $k \geq 1$, we have
\begin{align*}
PH_0(I^k) &\cong (0,1)^{2^k - 1} \oplus (0,\infty),\\
PH_1(I^k) &\cong (1,2)^{k \cdot 2^{k-1} - (2^k - 1)}.
\end{align*}
Moreover, $PH_1(I^k)$ is generated by the collection of images of the induced maps $PH_1(I^2) \rightarrow PH_1(I^k)$ of the coordinate inclusions $I^2 \rightarrow I^k$. Precisely, we have
$$PH_1(I^k) = \langle \bigcup_{\varphi \in C(2,k)} \varphi_*(PH_1(I^2)) \rangle.$$
\end{prp}
\begin{proof}
Recall from Remark~\ref{rmk:vietorisrips} that for $t \in \mathbb{R}_+$ there is an isomorphism $PH^{(t)}_*(I^k) \cong H_*(V(I^k,t))$, where $PH^{(t)}_*(I^k)$ is the homogeneous degree $t$ persistent homology and $V(I^k,t)$ is the Vietoris-Rips complex of $I^k$ at time $t$.

A simple calculation of $PH_0(I^k)$ follows from the observation that $V(I^k,0)$ is a simplicial complex consisting of $2^k$ points and no higher dimensional simplices, and $V(I^k,1)$ is connected. Alternatively, $PH_0(I^k)$ can be computed inductively using the K\"unneth Theorem~\ref{thm:kunnethmain} for metric spaces applied to $I^k = I \times I^{k-1}$, together with the observation $PH_0(I) \cong (0,1) \oplus (0,\infty)$ and the formulas for tensor product and $Tor$ of bars in Proposition~\ref{prp:tensortorbars}.

We now compute $PH_1(I^k)$. The claimed formula holds for $k = 1$ since $PH_1(I) = 0$. We now prove the formula for $k > 1$, assuming inductively that it holds for $k -1$. Indeed, by Theorem~\ref{thm:kunnethmain} applied to $X = I$ and $Y = I^{k-1}$ together with the formulas for tensor product and $Tor$ of bars in Proposition~\ref{prp:tensortorbars}, we compute
\begin{align*}
PH_1(I^k) &\cong (PH_0(I) \otimes PH_1(I^{k-1})) \oplus Tor_1(PH_0(I),PH_0(I^{k-1}))\\
&\cong (1,2)^{(k-1)\cdot 2^{k-1} - (2^k -2)} \oplus (1,2)^{2^{k-1}-1}\\
&= (1,2)^{k \cdot 2^{k-1} - (2^k - 1)},
\end{align*}
as claimed.

It remains to prove that $PH_1(I^k)$ is generated by the images of the induced maps of the coordinate inclusions $I^2 \rightarrow I^k$. From the computed formula for $PH_1(I^k)$, we see that it is generated by homogeneous degree $1$ elements. Hence it suffices to show that $PH_1^{(1)}(I^k)$ is generated by the images of the induced maps $PH_1^{(1)}(I^2) \rightarrow PH_1^{(1)}(I^k)$. This is the same as showing that $H_*(V(I^k,1))$ is generated by the images of the induced maps $H_*(V(I^2,1)) \rightarrow H_*(V(I^k,1))$.

The simplicial complex $V(I^k,1)$ is easily describable. The $0$-simplices are all the points in $I^k$. The $1$-simplices are all pairs $\{x,y\}$ such that $d_{I^k}(x,y) = 1$. There is a canonical identification $I^k \cong \mathbb{F}_2^k$ where $\mathbb{F}_2 = \{0,1\}$ is the finite field with $2$ elements. Under this identification, the set of $1$-simplices in $V(I^k,1)$ is given by all pairs $\{x,x+e_i\}$ for $x \in I^k$ and $1 \leq i \leq k$, where $e_i \in \mathbb{F}_2^k$ is the $i$-th standard basis vector. We claim that there are no higher dimensional simplices in $V(I^k,1)$. Indeed, suppose for contradiction that $\{x,x+e_i,y_1,\ldots,y_n\} \subset I^k$ is a simplex in $V(I^k,1)$ for some $n \geq 1$. Then $d_{I^k}(y_1,x) = 1$ implies $y_1 = x+e_j$ for some standard basis vector $e_j$. If $j = i$ then $y_1 = x+e_i$, which is not allowed, and if $j \neq i$ then $d_{I^k}(y_1,x+e_i) = 2$, which is not allowed in $V(I^k,1)$.

We identify the Vietoris-Rips complexes with their geometric realizations. The complex $V(I^2,1)$ is homeomorphic to $S^1$. So, to show that $H_1(V(I^k,1))$ is generated by the induced maps $H_1(V(I^2,1)) \rightarrow H_1(V(I^k,1))$ of the coordinate inclusions $I^2 \rightarrow I^k$, it suffices to show that if we glue a disk $D^2$ along the induced map on Vietoris-Rips $\varphi_V : \partial D^2 \cong V(I^2,1) \rightarrow V(I^k,1)$ for every $\varphi \in C(2,k)$, then every loop in $V(I^k,1)$ is nullhomotopic in this new space. Precisely, we form the space
$$X(k) := V(I^k,1) \cup \bigcup_{\varphi \in C(2,k)} D^2,$$
where for each $\varphi \in C(2,k)$ we identify $\partial D^2 \cong V(I^2,1)$ with its image $\varphi_V(V(I^2,1)) \subset V(I^k,1).$

It suffices to show that every loop in $V(I^k,1)$ is nullhomotopic in $X(k)$. This clearly holds for $k = 1$ since $X(k) = V(I^k,1) = [0,1]$ is contractible.

We proceed by induction. Let $k > 1$ and assume that every loop in $V(I^{k-1},1)$ is nullhomotopic in $X(k-1)$.

Observe from the above description of the $1$-simplices in $V(I^k,1)$ that we have an inclusion
\begin{align*}
V(I^k,1) &= (\{0\} \times V(I^{k-1},1)) \cup ([0,1] \times I^{k-1}) \cup (\{1\} \times V(I^{k-1},1))\\
&\subset  [0,1] \times V(I^{k-1},1).
\end{align*}
Moreover, there is an inclusion
$$[0,1] \times V(I^{k-1},1) \hookrightarrow X(k)$$
defined in the following way. Every $1$-simplex $\s$ in $V(I^{k-1},1)$ has boundary points $\{x,x+e_i\}$ for some $x \in I^{k-1}$ and standard basis vector $e_i \in I^{k-1}$, and so together with the first coordinate of $I^k = I \times I^{k-1}$ the simplex $\s$ determines a coordinate inclusion $\varphi_{\s} : I^2 \rightarrow I^k$, where the $2$ chosen coordinates of $I^k$ are the first one and the coordinate corresponding to $e_i \in I^{k-1}$, and the $\xi \in I^{k-2}$ is given by $x$ in the other coordinates. Then the inclusion $[0,1] \times V(I^{k-1},1) \hookrightarrow X(k)$ is defined by identifying $[0,1] \times \s$ with the copy of $D^2$ corresponding to $\varphi_{\s}$.

Consider a loop $S^1 \rightarrow V(I^k,1)$. This loop is homotopic in ${[0,1] \times V(I^{k-1},1)}$ to a loop in $\{0\} \times V(I^{k-1},1)$. After gluing on the disks corresponding to the coordinate inclusions $I^2 \rightarrow I^k$ that land in $\{0\} \times I^{k-1}$, our loop is in the space $\{0\} \times X(k-1)$, where it is nullhomotopic by the induction hypothesis. This completes the proof.
\end{proof}

\begin{prp} \label{prp:ph2lemma}
For $k \geq 1$, we have $PH_2(I^k) = 0$.
\end{prp}
\begin{proof}
By direct computation, one can check that the result holds for $k \leq 3$. We now proceed by induction. Let $k > 3$ and assume $PH_2(I^{k-1}) = 0.$ By Theorem~\ref{thm:kunnethmain}, the map ${\i_2}_* : PH_2(I,I^{k-1}) \rightarrow PH_2(I^k)$ is surjective, so it suffices to show that ${\i_2}_* = 0$.

Let $\varphi : I^2 \rightarrow I^{k-1}$ be any coordinate inclusion. Note that $PH_n(I) = 0$ for $n > 0$. Then by naturality of the short exact sequence \eqref{eq:classicalkunnethses} and the long exact sequence from Proposition~\ref{prp:les}, we have a commutative diagram
\[
  \begin{tikzcd}
    PH_2(I^3) \arrow{d} & \arrow{l}{{\i_2}_*} PH_2(I,I^2) \arrow{r}{\alpha}\arrow{d} & Tor_1(PH_0(I),PH_1(I^2)) \arrow{d}{{id_I}_* \otimes \varphi_*} \\
    PH_2(I^k) & \arrow{l}{{\i_2}_*} PH_2(I,I^{k-1}) \arrow{r}{\alpha} & Tor_1(PH_0(I),PH_1(I^{k-1})).  \end{tikzcd}
\]
Moreover, since $PH_n(I) = 0$ for $n > 0$ and since by induction we have $PH_2(I^{k-1}) = 0$, it follows from \eqref{eq:classicalkunnethses} that both maps labeled $\alpha$ are isomorphisms.

Since $PH_0(I) \cong (0,1) \oplus (0,\infty)$, by Proposition~\ref{prp:tensortorbars} we have 
$$Tor_1(PH_0(I),PH_1(I^{k-1})) \cong Tor_1((0,1),PH_1(I^{k-1})).$$
A free resolution of the bar $(0,1)$ is given by
$$0 \rightarrow (1,\infty) \rightarrow (0,\infty) \rightarrow (0,1) \rightarrow 0.$$
Now consider the bar $(1,2)$. Dropping the $(0,1)$ term from the free resolution and tensoring with $(1,2)$ yields (by Proposition~\ref{prp:tensortorbars}) the sequence
$$0 \rightarrow (2,3) \rightarrow (1,2) \rightarrow 0.$$
Observe in particular that the map $(2,3) \rightarrow (1,2)$ is necessarily zero. Hence, since $PH_1(I^{k-1})$ is a direct sum of bars $(1,2)$ by Proposition~\ref{prp:ph1lemma}, we have
$$Tor_1((0,1),PH_1(I^{k-1})) = (1,\infty) \otimes PH_1(I^{k-1}).$$
Hence we have shown that we have an isomorphism
$$
Tor_1(PH_0(I),PH_1(I^{k-1}))  \cong (1,\infty) \otimes PH_1(I^{k-1}).
$$
Under this isomorphism, the map ${id_I}_* \otimes \varphi_*$ in the commutative diagram above corresponds to the map $id_{(1,\infty)} \otimes \varphi_*$ in the following commutative diagram, where we have also used $PH_2(I^3) = 0$.
\[
  \begin{tikzcd}
    0 \arrow{d} & \arrow{l} PH_2(I,I^2) \arrow{r}{\cong}\arrow{d} & (1,\infty) \otimes PH_1(I^2) \arrow{d}{id_{(1,\infty)} \otimes \varphi_*} \\
    PH_2(I^k) & \arrow{l}{{\i_2}_*} PH_2(I,I^{k-1}) \arrow{r}{\cong} & (1,\infty) \otimes PH_1(I^{k-1}).  \end{tikzcd}
\]

Now, there is a generating set of $PH_2(I,I^{k-1})$ corresponding under the isomorphism with $(1,\infty) \otimes PH_1(I^{k-1})$ to the generating set of $PH_1(I^{k-1})$ described in Proposition~\ref{prp:ph1lemma} (tensored with the generator of $(1,\infty)$), which consists of the images of the induced maps $\varphi_* : PH_1(I^2) \rightarrow PH_1(I^{k-1})$ of the coordinate inclusions $\varphi : I^2 \rightarrow I^{k-1}$. For any such coordinate inclusion $\varphi$, commutativity of the above diagram shows that the image of $(1,\infty) \otimes PH_1(I^2)$ in $PH_2(I^k)$ is $0$. Hence ${\i_2}_* = 0$.
\end{proof}

\section{Appendix: simplicial sets} \label{sec:simplicialsets}

We review basic notions about simplicial sets.

Let $\ul{\Delta}$ denote the \emph{simplex category}. It has objects $[n] = \{0,\ldots,n\}$ for $n \geq 0$ and morphisms the weakly order preserving maps, i.e. $\varphi : [m] \rightarrow [n]$ is a morphism in $\ul{\Delta}$ if it satisfies $\varphi(0) \leq \varphi(1) \leq \cdots \leq \varphi(m).$

A \emph{simplicial set} is a contravariant functor ${\bf X} : \ul{\Delta} \rightarrow Set$. Concretely, ${\bf X}$ consists of a set ${\bf X}_n$ for each $n \geq 0$ together with a set map ${\bf X}_{\varphi} : {\bf X}_n \rightarrow {\bf X}_m$ for each morphism $\varphi : [m] \rightarrow [n]$ in $\ul{\Delta}$ such that the collection of these maps satisfy the usual functorial properties.

The category of simplicial sets $sSet$ has simplicial sets as objects and natural transformations as morphisms. The morphisms are called \emph{simplicial maps}.

There are special morphisms in $\ul{\Delta}$ called the \emph{coface maps}
\begin{align*}
d^{\Delta}_i : [n] &\rightarrow [n+1]\\
j &\mapsto j \text{ if } j < i,\\
j &\mapsto j+1 \text{ if } j \geq i
\end{align*}
for $0 \leq i \leq n+1$ and the \emph{codegeneracy maps}
\begin{align*}
s^{\Delta}_i : [n] &\rightarrow [n-1]\\
j &\mapsto j \text{ if } j \leq i,\\
j &\mapsto j-1 \text{ if } j > i.
\end{align*}
for $0 \leq i \leq n-1$. These generate the morphisms in $\ul{\Delta}$ in the sense that any morphism can be written as the composition of coface and codegeneracy maps.

Given ${\bf X} \in sSet$, the maps ${\bf X}_{d^{\Delta}_i}$ are called \emph{face maps} and the maps ${\bf X}_{s^{\Delta}_i}$ are called \emph{degeneracy maps}. Often we denote ${\bf X}_{d^{\Delta}_i}$ by $d_i$ and ${\bf X}_{s^{\Delta}_i}$ by $s_i$ when the simplicial set ${\bf X}$ is implicit.

\bibliographystyle{amsplain}
\bibliography{../../../references}
\end{document}